%% file: aHTCovariantReps.tex
\documentclass[a4paper, 12pt, reqno]{amsart}
\input{preamble}

\title[Covariant representations of actions of inverse semigroups]{Covariant representations of actions of inverse semigroups:  a new approach to the reduced and essential crossed-product C*-algebras.}

\author[an Huef]{Astrid an Huef}
\author[Tolich]{Ilija Tolich}
\address[A.~an Huef and I.~Tolich]{School of Mathematics and Statistics, Victoria University of Wellington, PO Box 600, Wellington 6140, Aotearoa New Zealand.}
\email[A.~an Huef]{\href{mailto:astrid.anhuef@vuw.ac.nz}{astrid.anhuef@vuw.ac.nz}}
\email[I.~Tolich]{\href{mailto:ilija.tolich@vuw.ac.nz}{ilija.tolich@vuw.ac.nz}}

\date{\today}



\thanks{This research was supported by  Marsden grant 24-VUW-050 of the Royal Society of New Zealand. This paper arose from work on \cite{ACaHMT}, and we would like to thank our co-authors  Becky Armstrong, Lisa Clark and Diego Mart\'inez.
}

\begin{document}

\begin{abstract}
We consider an action of an inverse semigroup on a C*-algebra $A$ and use it to construct a groupoid of germs with unit space the spectrum of $A$. Motivated by the representation theory of C*-algebras of groupoids, we construct a concrete family of covariant representations for the action. We use this family to give new definitions of the reduced and essential crossed product C*-algebras that avoid, respectively, passing to the double commutant and local multiplier algebra of $A$. Our reduced crossed product is isomorphic to the one defined by Exel, Buss and Meyer, and when the inverse semigroup is quasi-countable  our essential crossed product is isomorphic to the one defined by Kwa\'sniewski  and Meyer. 
\end{abstract}

\maketitle

\section{Introduction}
 Let $\alpha$ be an action by automorphisms of a locally compact group $G$   on a C*-algebra $A$. The maximal crossed product $A\rtimes_\full G$ is the C*-algebra universal for covariant representations $(\pi, U)$ of the system $(A, G,\alpha)$ \cite{Raeburn-crossed-products, Williams-x-products}.  Thus $A\rtimes_\full G$  is generated by  a universal covariant representation and every non-degenerate representation is the integrated form $\pi\times U$ of a covariant representation $(\pi, U)$. Concretely, $A\rtimes_\full G$ is the completion of $C_c(G, A)$, the continuous compactly supported functions on $G$ with values in $A$, in a universal norm where for  $f\in C_c(G,A)$ the norm of $f$ is the supremum  over all covariant representations of $\|\pi\times U(f)\|$.    The reduced crossed product $A\rtimes_\red G$ is a quotient  of $A\rtimes G$ obtained by completing the image of $C_c(G,A)$ acting on the space of a  regular representation  on a concrete Hilbert space of square-integrable functions. If the action is minimal and topologically free, then $A\rtimes_\red G$ is simple by \cite[Corollary on page~122]{Archbold-Spielberg}.

A partial automorphism of a  C*-algebra $A$ is an isomorphism between two ideals of $A$ \cite{Exel-1994}.
In \cite{Sieben-JAMS-1997}, Sieben defined   an action $\alpha$ of a discrete inverse semigroup $S$ with an identity on a C*-algebra $A$ by partial automorphisms, and developed an appropriate notion of covariant representations and then defined a universal crossed product C*-algebra $A\times_\full S$. This theory was extended by Exel \cite{Exel2008} and Paterson  \cite{Paterson-groupoid-book}  to allow for non-unital inverse semigroups. Over the years, the definition of the $*$-subalgebra that is completed to $A\times_\full S$ has been refined  in \cite{BussExelMeyer2017, BussMartinez2023, KwasniewskiMeyer2021} to the one we denote by $A\rtimes_\alg S$.  In particular, as pointed out in \cite[Remark~3.17]{KwasniewskiMeyer2021}, there is an injective embedding of $A\rtimes_\alg S$ into $A\rtimes_\full S$.

The reduced crossed product $A\rtimes_\red S$, a quotient of $A\rtimes_\full S$, was first introduced by Exel in \cite[\S~8.3]{Exel2011} although in the greater generality of a  reduced cross-sectional algebra of a Fell bundle over $S$. The construction was revisited by Buss, Exel and Meyer in \cite[Definition~4.1]{BussExelMeyer2017} with an equivalent definition using  a weak expectation   \[\ER\colon A\rtimes_\full S \to  A''\] into the enveloping von Neumann algebra of $A$, and  in \cite{KwasniewskiMeyer2021} was shown to be  the quotient of $A\rtimes_\full S$ by the ``nucleus''  $\{a\in A\rtimes_\full S: \ER(a^*a)=0\}$.  
It is shown in   \cite{BussExelMeyer2017} that the norm of $A\rtimes_\red S$ is the supremum of  norms in $\Ind\pi''$, where $\pi''$ is the extension to $A''$ of a representation $\pi$ of $A$; the induction is via  a right-Hilbert $(A''\rtimes_\red S)$-$A''$ bimodule.

Khoshkam and Skandalis in \cite{KhoshkhamSkandalis2004}  also defined a   universal crossed product $A\rtimes_{\mathrm{KS}} S$.  They used a weaker notion of covariant representation so that Sieben's crossed product is a quotient of $A\rtimes_{\mathrm{KS}} S$.  Furthermore, they  defined a reduced crossed product   using a concrete family of covariant representations to obtain a quotient of $A\rtimes_{\mathrm{KS},\red} S$, which, from a philosophical point of view, is how a reduced crossed product should be defined. Khoshkham and Skandalis' reduced crossed product $A\rtimes_{\mathrm{KS},\red} S$ is distinct from $A\rtimes_{\red} S$: for example, when $S$ is $E*$-unitary without a $0$ element,  then $A\rtimes_{\red} S$ is a quotient of $A\rtimes_{\mathrm{KS},\red} S$ (see \cite[Proposition 7.5]{BedosNorling2017}). 
For other literature examining the Khoshkham-Skandalis construction see \cite{BedosNorling2017,Sekiyama2026,Sundar2018}.

The first objective of this paper is to show explicitly that $A\rtimes_\red S$ is a quotient of $A\rtimes_\full S$ under a family of concrete covariant representations. 
We use  an induced action $\beta$ of $S$ on the (possibly non-Hausdorff) spectrum $\hat{A}$ of $A$ by partial homeomorphisms to construct a groupoid of germs $G(\hat{A},S,\beta)$. Using the case when $A$ is commutative for inspiration,  we  obtain  a family   $\{(\rho_\pi,\sigma_\pi):\pi\in \hat{A}\}$ of covariant representations that   mimic   left regular representations of the groupoid C*-algebra.  We show that for $a\in A\rtimes_\alg S$ the supremum of $\|\rho_\pi\times \sigma_\pi(a)\|$ gives a  norm on $A\rtimes_\alg S$ and define a C*-algebra  $\Red(A,S,\alpha)$ as  the completion of $A\rtimes_\alg S$  in that norm.  In \Cref{sec reconcile} we prove that $\Red(A, S,\alpha)$ equals $A\rtimes_\red S$; equivalently we prove that the nucleus of $\ER$ is $\cap _{\pi\in\hat A}\ker(\rho_\pi\times\sigma_\pi)$ (see \Cref{thm reconciliation for reduced}).  Furthermore, we show that when $\pi\in \hat A$,  the restriction of $\Ind\pi''$ to $A\rtimes_\red S$ is unitarily equivalent to $\rho_\pi\times \sigma_\pi$. 
The advantage of our definition of $A\rtimes_\red S$ is that it only uses covariant representations and avoids the technicalities of the weak conditional expectation taking values in the double commutant.

Another important quotient of $A\rtimes_\full S$, defined by defined by Kwa\'sniewski and Meyer in \cite{KwasniewskiMeyer2021}, is  the essential crossed product $A\rtimes_\ess S$. It is by definition the  quotient of $A\rtimes_\full S$ by the nucleus of a generalised expectation \[\EL\colon  A\rtimes_\full S\to \Mloc(A)\] into the local multiplier algebra $\Mloc(A)$ of $A$.  If the action of $S$ on $A$ is aperiodic, then $A\rtimes_\ess S$ is the unique quotient of $A\rtimes_\full S$ that detects ideals in $A$ \cite[Theorem~6.5]{KwasniewskiMeyer2021}. Furthermore, if the action of $S$ on $A$ is aperiodic and minimal, then $A\rtimes_\ess S$ is simple \cite[Theorem~6.6]{KwasniewskiMeyer2021}. 

As noted in \cite[\S2]{KKLRU}, this work of Kwa\'sniewski and Meyer also gives a systematic and unifying approach to the definition of the essential C*-algebra $C^*_\ess(G)$ of an \'etale groupoid $G$, developed over years in \cite{Clark-Exel-Pardo-Sims-Starling, ExelPitts2022, Exel-2011-PAMS, KhoshkhamSkandalis2002}.
The key idea behind this realisation is that the set $\Bis(G)$ of    open bisections of $G$ is an inverse semigroup which acts naturally on the unit space $X$ of $G$.   The crossed product $C_0(X)\rtimes_\full \Bis(G)$ is isomorphic to the full groupoid C*-algebra $C^*(G)$ by \cite[Proposition~9.8]{Exel2008}, see also \cite[Theorem~3.3.1]{Paterson-groupoid-book}, \cite[Theorem~8.1]{Quigg-Sieben}.  (The reduced and essential groupoid C*-algebras of $G$ also coincide, respectively,  with  $C_0(X)\rtimes_\red \Bis(G)$ and   $C_0(X)\rtimes_\ess \Bis(G)$ \cite[Theorem 8.1]{ACaHMT}.)

The second objective  of this paper is to characterise the nucleus of $\EL$ in terms of the representations $\rho_\pi\times \sigma_\pi$ above, which we achieve  when $S$ is quasi-countable. Consequently, if $S$ is quasi-countable, the essential crossed product is a completion (after modding out by vectors of norm zero) of $A\rtimes_\alg S$ in a seminorm defined using only the family $\{(\rho_\pi,\sigma_\pi):\pi\in \hat{A}\}$. When $S$ is not  quasi-countable, we can still characterise the nucleus of $\EL$ using the family $\{\rho_\pi\times \sigma_\pi:\pi\in \hat A\}$.  Again, the advantage of our construction is that it avoids the technicalities of the generalised expectation taking values in the local multiplier algebra. 

\subsection*{Outline} 
We  begin in \Cref{sec prelim} by presenting the  key definitions and results about an action of an inverse semigroup $S$ on a C*-algebra $A$, covariant representations of this system and the  definition of the maximal crossed product.  We also introduce the induced action $\beta$ of $S$ on the spectrum  $\hat A$ of $A$.
In \Cref{sec reg reps} we use a groupoid of germs $G(\hat{A},S,\beta)$ to construct and then study our family of concrete covariant representations $\{(\rho_\pi,\sigma_\pi):\pi\in \hat{A}\}$. At the end of \Cref{sec reg reps}  we define two C*-algebras $\Red(A,S,\alpha)$ and $\Ess(A,S,\alpha)$  using the
representations $(\rho_\pi,\sigma_\pi)$; in \Cref{sec reconcile} we show that $\Red(A,S,\alpha)=A\rtimes_\red S$, and that $\Ess(A,S,\alpha)=A\rtimes_\ess S$ when $S$ is quasi-countable.

\section{Preliminaries}\label{sec prelim}
We begin by recalling some preliminary results that we will use throughout the paper. 

\subsection{The spectrum of a C*-algebra}
Let $A$  be a C*-algebra and let $\pi$ be an irreducible representation of $A$;  we use the
same symbol $\pi$ to denote the unitary equivalence class of $\pi$ in the spectrum  $\hat A$ of $A$.

Let $I$ be an ideal of a C*-algebra $A$ and let $\pi\colon I\to B(H_\pi)$ be a  non-degenerate representation.  Then there is a unique representation $\bar{\pi}\colon A\to B(H_\pi)$ such that $\bar{\pi}|_I=\pi$.
Furthermore, if $\pi$ is injective and $I$ is an essential ideal of $A$, then $\bar{\pi}$ is injective. See, for example, \cite[Proposition~2.50]{RaeburnWilliams1998}.  There is a  homeomorphism $\pi\mapsto \pi|_{I}$  from the open subset $\{\pi \in \hat{A} : \pi|_I\ne 0\}$  of $\hat A$ onto $\hat I$, with inverse given by  the unique extension $\pi \mapsto \bar{\pi}$ (see, for example, \cite[Proposition~A.27]{RaeburnWilliams1998}).

\subsection{Actions of inverse semigroups}

For inverse semigroups and actions of inverse semigroups on C*-algebras, we follow \cite{ACaHMT, BussMartinez2023, Exel2008, Paterson-groupoid-book, Sieben-JAMS-1997}.
An \emph{inverse semigroup} is a semigroup $S$ such that every $s \in S$ has a unique \emph{inverse} $s^* \in S$ such that $ss^*s = s$ and $s^*ss^* = s^*$.

Let $S$ be an inverse semigroup. Then $S$ has a 
partial order given by \[s \le t \iff s = ss^*t\] for $s, t \in S$. An $e \in S$ is called an \emph{idempotent} if $e^2 = e$. We write $E(S)$ for the collection of idempotents in $S$. For $e \in E(S)$, we have $e^* = e$. For  $e, f \in E(S)$, we have $e \le f \iff e = ef$, and it follows that \[E(S) = \{ ss^* : s \in S \}.\] Moreover, $E(S)$ is a commutative subsemigroup of $S$, and it follows that $s \le t \iff s = ts^*s$ for $s, t \in S$. Since  $E(S)$ is commutative, we have $ses^*, s^*es \in E(S)$ for all $s \in S$ and $e \in E(S)$. If $S$ has an identity element $1$, then we say that $S$ is \emph{unital}.  If $S$ is unital, then $E(S) = \{s \in S : s \le 1\}$.

Let $X$ be a set and let $\II(X)$  be the set of partial bijections on $X$; that is, bijections $f\colon \dom(f) \to \ran(f)$ where $\dom(f)$ and $\ran(f)$ are subsets of $X$.
Then  $\II(X)$ is an inverse semigroup with respect to composition ``where it makes sense''; that is, $f \circ g$ is defined by $(f \circ g)(x) \coloneqq f(g(x))$ for all $x \in g^{-1}\big(\!\ran(g) \cap \dom(f)\big)$. The inverse of $f \in \II(X)$ is  $f^*\colon \ran(f) \to \dom(f)$ given by $f^*(x) \coloneqq f^{-1}(x)$. Note that for each $f \in \II(X)$, $f^*\circ f$ is the identity map on $\dom(f)$, and $f\circ f^*$ is the identity map on $\ran(f)$. 

\begin{defn} 
\label{defn: C*-algebra action}
An \emph{action} $\alpha\colon S \curvearrowright A$ of an inverse semigroup $S$ on a C*-algebra $A$ is a homomorphism  $\alpha\colon S \to \II(A)$ of  inverse semigroups   such that
\begin{enumerate}
\item \label{item: alpha_s iso} for each $s \in S$, $\alpha_s$ is a C*-isomorphism and its domain $A_{s^*s}$ is an ideal of $A$; and
\item \label{item: ideals dense} the linear span of $\{A_{s^*s}:s\in S\}$ is dense in $A$.
\end{enumerate}
\end{defn}

Let $\alpha\colon S \curvearrowright A$ be an action of an inverse semigroup $S$ on a C*-algebra $A$. For $s, t \in S$, define
\[
I_{s,t} \coloneqq \clsp\{ a \in A : a \in A_{v^*v} \text{ for some } v \in S \text{ with } v \le s, t \}.
\]

The properties of $I_{s,t}$ are investigated in detail in \cite[\S2]{ACaHMT}. In the following lemma we summarise the properties that we will use frequently.

\begin{lemma} \label{lem: I_st properties}
Let $\alpha\colon S \curvearrowright A$ be an action of an inverse semigroup $S$ on a C*-algebra $A$. Fix $s, t \in S$. Then
\begin{enumerate}
    \item  $I_{s,t}$ is an ideal of $A$;
    \item $I_{s,t} = I_{t,s} \subseteq A_{s^*s} \cap A_{t^*t}$ and $I_{s,s}=A_{s^*s}$;
    \item  if $e\in E(S)$, then $I_{e,s}=I_{e,s^*}$;
    \item $\alpha_s(a) = \alpha_t(a)$ for all $a \in I_{s,t}$.
\end{enumerate}
\end{lemma}

\subsection{Covariant representations and the maximal crossed product}

We   recall the construction of the
 universal $C^*$-algebra $A\rtimes_{\max} S$ associated to  $\alpha\colon S \curvearrowright A$  from, for example, \cite[Definition~4.4]{Sieben-JAMS-1997}, \cite[page~233]{Exel2008} and \cite[page~232]{BussExelMeyer2017}. While these constructions  complete different $*$-subalgebras, the resulting C*-algebras are the same.  
We start with the definition of covariant representation of $\alpha\colon S \curvearrowright A$ as in \cite[Definition~3.4]{Sieben-JAMS-1997} and \cite[Definition~9.4]{Exel2008}. 

\begin{defn}
\label{defn covariant pair}
Let $\alpha\colon S \curvearrowright A$ be an action of an inverse semigroup $S$ on a C*-algebra $A$. A \emph{covariant representation} of $\alpha\colon S \curvearrowright A$ on a Hilbert space $H$ is a pair $(\pi,\sigma)$, where $\pi\colon A\to B(H)$ is a non-degenerate representation  and $\sigma\colon S\to B(H)$ such that for all $s,t\in S$ and $e\in E(S)$ we have
\begin{enumerate}
\item \label{defn covariant pair 1}$\sigma_{st}=\sigma_s\sigma_t$ and 
$\sigma_{s^*}=\sigma_s^*$;
\item \label{defn covariant pair 2}$\pi(\alpha_s(a))=\sigma_s\pi(a)\sigma_{s^*}$ for $a\in A_{s^*s}$; and 
\item \label{defn covariant pair 3}
$\pi(A_e)H=\sigma_e(H)$  (note that both sides are closed; for the left-hand-side, see, for example,  \cite[Theorem~32.22 on page~268]{Hewitt-Ross-2}).
\end{enumerate}
\end{defn}

The following properties of covariant representations are used repeatedly in our proofs.

\begin{lemma}\label{lem cov rep basics}
    Let $\alpha\colon S \curvearrowright A$ be an action of an inverse semigroup $S$ on a C*-algebra $A$ and let $(\rho,\sigma)$ be a covariant representation.
    \begin{enumerate}
        \item\label{lem cov rep basics item 1} Let $e\in E(S)$ and $a\in A_e$. Then $\rho(a)\sigma(e)=\rho(a)=\sigma(e)\rho(a)$.
        \item\label{lem cov rep basics item 2} Let $s\in S$ and $a\in A_{s^*s}$. Then $\sigma(s)\rho(a)=\rho(\alpha_{s}(a))\sigma(s)$. 
    \end{enumerate}
\end{lemma}
\begin{proof}
Since $\alpha_e(a)=a$ and $(\rho, \sigma)$ is covariant, we have $\sigma(e)\rho(a)\sigma(e)=\rho(\alpha_e(a))=\rho(a)$. Since  $\sigma(e)^2=\sigma(e)$ we have 
$\rho(a)\sigma(e)=\sigma(e)\rho(a)\sigma(e)^2=\rho(a)$. 
Similarly $\sigma(e)\rho(a)=\rho(a)$.

Now $\rho(\alpha_{s}(a))\sigma(s)=\sigma(s)\rho(a)\sigma(s^*)\sigma(s)=\sigma(s)\rho(a)\sigma(s^*s)=\sigma(s)\rho(a)$. 
\end{proof}

 We consider the vector space of functions
\[
C_c(S,A)\coloneqq \{ f\colon S \to A \,:\, f(s) \in A_{ss^*} \text{ and } f(s) = 0 \text{ for all but finitely many } s \in S \}.
\]
(We use the notation $C_c$ to denote continuous compactly supported functions, but since the domain $S$ is discrete, these are actually finitely supported functions.) We equip $C_c(S,A)$ with pointwise addition, scalar multiplication, and multiplication and adjoint given by convolution and involution, respectively: for $f, g \in C_c(S,A)$ and $s \in S$,
\[
(f * g)(s) \coloneqq \sum_{\substack{t,u \in S, \\ s = tu}} \alpha_t\big(\alpha_{t^*}(f(t)) \, g(u) \big) \quad \text{ and } \quad f^*(s) \coloneqq \alpha_s\big(f(s^*)^*\big).
\]
For $s \in S$ and $a_s \in A_{ss^*}$, we define $a_s\delta_s\colon S \to A$ by
\[
a_s\delta_s(t) \coloneqq \begin{cases} a_s & \text{if } t = s \\ 0 & \text{if } t \ne s. \end{cases}
\]
Then for $s, t \in S$, $a_s \in A_{ss^*}$, and $b_t \in A_{tt^*}$, we have
\[
(a_s\delta_s) * (b_t\delta_t) = \alpha_s\big(\alpha_{s^*}(a_s) \, b_t\big) \delta_{st} \quad \text{ and } \quad (a_s\delta_s)^* = \alpha_{s^*}(a_s^*) \delta_{s^*}.
\]

A covariant representation $(\pi, \sigma)$ of $\alpha\colon S \curvearrowright A$ on a Hilbert space $H$ gives rise to a non-degenerate representation $\pi\times\sigma$ of $C_c(S,A)$,  called the \emph{integrated form of $(\pi, \sigma)$},  such that 
\begin{equation}\label{eq cov rep}
\pi\times\sigma \big(\sum a_s\delta_s\big)=\sum\pi(a_s)\sigma_s.
\end{equation}
 
Let $s, t \in S$ with $s \le t$ and $a \in A_{ss^*}$. Then $s^* \le t^*$ and so $A_{ss^*} \subseteq A_{tt^*}$. Now  $a\delta_s$ and $a\delta_t$ are both elements of $C_c(S,A)$ and it is natural to identify $a\delta_s$ with $a\delta_t$ because $a\delta_s-a\delta_t$ is in the kernel of the integrated form of every covariant representation.  To better understand these kernels, we  consider the ideals 
\begin{align*}
\JJ_\alpha&\coloneq \vecspan\{a\delta_s-a\delta_v:v,s\in S, v\le s, a\in A_{vv^*}\} \text{\ and }\\
\II_\alpha&\coloneq\vecspan\{a\delta_s-a\delta_r:s,r\in S, a\in I_{s^*,r^*}\}
\end{align*}
of $C_c(S,A)$. 
The relations generating $\JJ_\alpha$ were first introduced  in \cite[Lemma~4.5]{Sieben-JAMS-1997} and  $\II_\alpha$ was introduced in \cite[\S2]{BussExelMeyer2017} for an action of $S$ on $A$ by Hilbert modules. It follows from \cite[Proposition~2.9]{BussExelMeyer2017} that $\II_\alpha$ is contained in  the kernel of the integrated form of every covariant representation; because the notation in \cite{BussExelMeyer2017} differs from ours we prove this in our setting in  the next lemma. 

\begin{lemma}
\label{lem: I_alpha and cov reps}
    Let $\alpha\colon S \curvearrowright A$ be an action of an inverse semigroup $S$ on a C*-algebra $A$. Let $(\rho,\sigma)$ be a covariant representation on a Hilbert space $H$. Then \[\JJ_\alpha\subseteq \II_\alpha\subseteq \ker (\rho\times \sigma).\]
\end{lemma}

\begin{proof}
   If $v,s\in S$ such that $v\le s$ , then  $ I_{s^*,v^*}=A_{vv^*}$. Thus $\JJ_\alpha\subseteq \II_\alpha$.

Next, fix a spanning element of $\II_\alpha$, that is, $a\delta_s-a\delta_r$  where $s,r\in S$ and $a\in I_{s^*,r^*}$.  By definition of $I_{s^*,r^*}$, we have  $a=\lim_{n\to\infty}b_n$, where each $b_n$ is a finite sum  $b_n=\sum_{u\in F_n} a_{u}$ where $u\le s^*,r^*$ and $a_{u}\in A_{u^*u}$.
Now 
\begin{align*}\rho\times\sigma(a\delta_s-a\delta_r)&=\rho(\lim_{n\to\infty}b_n)(\sigma(s)-\sigma(r))=\lim_{n\to\infty}\rho(b_n)(\sigma(s)-\sigma(r))
\end{align*}
where each \[\rho(b_n)(\sigma(s)-\sigma(r))=\sum_{u\in F_n} \rho(a_{u})(\sigma(s)-\sigma(r)).\]
Now we consider the summands $\rho(a_{u})(\sigma(s)-\sigma(r))$ where $u\in F_n$ and $u\leq s^*, r^*$.
 Recall that $u\le s^*,r^*$ if and only if $u^*\le s,r$.  Thus $u^*us=u^*=u^*ur$ and using the properties of a covariant representation we obtain
    \begin{align*}
        \rho(a_{u})(\sigma(s)-\sigma(r))&=\rho(\alpha_{u^*u}(a_{u}))(\sigma(s)-\sigma(r))
        \\
        &=\sigma(u^*u)\rho(a_{u})\sigma(u^*u)(\sigma(s)-\sigma(r))
        \\
        &=\sigma(u^*u)\rho(a_{u})(\sigma(u^*us)-\sigma(u^*ur))
        \\
        &=\sigma(u^*u)\rho(a_{u})(\sigma(u^*)-\sigma(u^*))
        =0
    \end{align*}
Thus each $\rho(b_n)(\sigma(s)-\sigma(r))=0$ and hence $\rho\times\sigma(a\delta_s-a\delta_r)=0$. 
It follows that $\II_\alpha\subseteq \ker (\rho\times\sigma)$.
\end{proof}

Indeed, we will see in \Cref{prop: reduced norm is a norm} below that $\II_\alpha=\cap\{\ker(\rho\times \sigma): (\rho,\sigma)\text{  is covariant}\}$.

\begin{defn}\label{defn: algebraic crossed product}
    Let $\alpha\colon S \curvearrowright A$ be an action of an inverse semigroup $S$ on a C*-algebra $A$. The \emph{algebraic crossed product} of $(A,S,\alpha)$ is  \[A\rtimes_\alg S\coloneq C_c(S,A)/_{\II_\alpha}.\]
\end{defn}

By \Cref{lem: I_alpha and cov reps}, the integrated form of every covariant representation factors through the algebraic crossed product.  

\begin{defn}
    Let $\alpha\colon S \curvearrowright A$ be an action of an inverse semigroup $S$ on a C*-algebra $A$.  For $a\in A\rtimes_\alg S$ define 
    \[\|a\|_\full=\sup\{\|\rho\times\sigma(a)\|: \rho\times\sigma \text{ is a covariant representation\}}. \]
    The \emph{maximal crossed product} $A\rtimes_\full S$ of $\alpha\colon S \curvearrowright A$ is the  C*-algebraic (hence Hausdorff) completion of $A\rtimes_\alg S$  with respect to $\|\cdot\|_\full$. 
\end{defn}

For every covariant representation  $(\rho,\sigma)$ on a Hilbert space $H$ there exists a representation $\rho\times \sigma\colon A\rtimes_\full S \to B(H)$ characterised by the formula  at \eqref{eq cov rep}. 
Furthermore, by \cite[Theorem~9.7]{Exel2008}, 
$(\rho,\sigma)\mapsto\rho\times\sigma$ is a bijection between the covariant representations of $\alpha\colon S \curvearrowright A$   
and the non-degenerate representations 
of  $A\rtimes_\full S$.

If $S$ is unital, 
then $A$ embeds in $C_c(S,A)$ and $A\rtimes_\alg S$ via $a\mapsto a\delta_1$. If $S$ is non-unital, then $A$ does not embed in $A\rtimes_\alg S$, but  $A$ still  embeds in $A\rtimes_\full S$. As noted in \cite{Exel2008,BussExelMeyer2017,BussMartinez2023}, for example,  we can adjoin an identity to $S$ without changing the C*-algebra. We now show how this works directly.

\begin{lemma}\label{lem A embeds}
    Let $\alpha\colon S \curvearrowright A$ be an action of an inverse semigroup $S$ on a C*-algebra $A$. Let $\tilde{S}\coloneq S\cup\{1\}$ and extend $\alpha$ to $\tilde\alpha$ by setting $\tilde\alpha_1=\id$.
    Then  $A\rtimes_\full S$ and $ A\rtimes_\full \tilde{S}$ are isomorphic and the map $a_s\mapsto a_s\delta_{s^*s}$, where $s\in S$ and  $a_s\in A_{ss^*}$, extends to an embedding of $A$ in $A\rtimes_\full S$.
\end{lemma}

\begin{proof}
Let $\Psi\colon A\rtimes_\full \tilde{S}\to B(H)$ be a faithful representation. By \cite[Theorem~9.7]{Exel2008} there exists a covariant representation $(\rho,\sigma)$ of $\tilde{\alpha}\colon \tilde{S}\curvearrowright A$ such that $\Psi=\rho\times \sigma$. Then $(\rho, \sigma|_S)$ is a covariant representation of $\alpha\colon S\curvearrowright A$ and $\rho\times\sigma|_S$ is isometric because $\rho\times\sigma$ is. 
Thus $\Phi\coloneq \Psi^{-1}\circ\rho\times \sigma|_S:A\rtimes_\full S\to A\rtimes_\full \tilde{S}$ is isometric and it suffices to show that it is onto. For this we show that the ranges of $\rho\times\sigma|_S$ and $\rho\times\sigma$ coincide. 
Since $\rho\times\sigma(A\rtimes_\full \tilde S)$ is generated by elements $\rho(a)\tilde\sigma(s)$ where $s\in\tilde S$ and $a\in A_{s^*s}$, it suffices to show that $\rho(a)\tilde\sigma(1)=\rho(a)\in \rho\times\sigma(A\rtimes_\full S)$ for $a\in A$.

Fix $a\in A$. Since the linear span of the ideals $A_{s^*s}$ is dense in $A$, $a$ is the limit of finite sums of the form $\sum_{s\in F_n} b_s$ where $b_s\in A_{ss^*}$.
By \Cref{lem cov rep basics} we have $\rho(b_s)=\rho(b_s)\sigma(ss^*)$ and thus
\[\rho(a)=\lim_{n\to\infty} \rho\Big(\sum_{s\in F_n} b_s\Big)=\lim_{n\to\infty} \sum_{s\in F_n} \rho(b_s)=\lim_{n\to\infty} \sum_{s\in F_n} \rho(b_s)\sigma(ss^*) \in \rho\times\sigma(A\rtimes_\full S).
\]
Thus $\Phi$ is an isomorphism as claimed, and for   $s\in S$ and $a_s\in A_{ss^*}$ we have \[\Phi(a_s\delta_s)=\Psi^{-1}\circ(\rho\times \sigma|_S)(a_s\delta_s)=\Psi^{-1}(\rho(a_s)\sigma(s))=a_s\delta_s.\]

Finally, the map $\iota\colon a\mapsto a\delta_1$ extends to an embedding  $A$ in  $A\rtimes_\full \tilde{S}$. Now $\Phi^{-1}\circ\iota$ is an embedding. Using \Cref{lem cov rep basics} again we have  $\Phi^{-1}\circ \iota(a_s)=\Phi^{-1}(a_s\delta_1)=\Phi^{-1}(a_s\delta_{ss^*})=a_s\delta_{ss^*}$
\end{proof}

\subsection{An induced action of the semigroup on the spectrum}
We will show that an action of an inverse semigroup $S$ on a C*-algebra $A$ induces an action on the spectrum $\hat{A}$.

\begin{defn} Let $X$ be a topological space. 
    An \emph{action} $\beta\colon S\curvearrowright X$ is an inverse semigroup homomorphism $\beta:S\to \II(X)$ such that 
    \begin{enumerate}
        \item for each $s\in S$, $\beta_s$ is a homeomorphism and its domain $U_{s^*s}$ is an open set of $X$; and
        \item $\cup_{s \in S} \, U_{s^*s}=X$. 
    \end{enumerate}
\end{defn}

\begin{prop}\label{prop: induced action on spectrum}
Let $\alpha\colon S \curvearrowright A$ be an action of an inverse semigroup $S$ on a C*-algebra $A$.
For  $s\in S$,   let
 \[U_{s^*s}\coloneqq\{\pi\in \hat{A}: \pi|_{A_{s^*s}}\neq 0\}\quad 
 \text{and}\quad   \beta_s(\pi)\coloneqq\overline{\pi\circ\alpha_{s^*}}.\]
Then  each 
 $\beta_s\colon U_{s^*s}\to U_{ss^*}$ 
is a homeomorphism, and  $\beta\colon S\curvearrowright \hat{A}$ is an action. 
\end{prop}
    \begin{proof} Let $s\in S$. 
        Then $U_{s^*s}$ is an open subset of $\hat{A}$ homeomorphic to $\hat A_{s^*s}$.
        Since $\alpha$ is an action, $\vecspan\{A_{s^*s}:s\in S\}$ is dense in $A$, which implies that $\cup_{s \in S} \, U_{s^*s}=\hat{A}$. 

        Let $\pi\in U_{s^*s}$.  Since $\alpha_{s^*}\colon A_{ss^*}\to A_{s^*s}$ is an isomorphism, $0\neq \pi\circ \alpha_{s^*}$ is an irreducible representation of $A_{s^*s}$, and its unique extension $\overline{\pi\circ \alpha_{s^*}}$ to $A$ is irreducible.
        Thus  $\beta_s(\pi )\in U_{ss^*}$ and  is well-defined. 

        Notice that $\overline{\pi\circ\alpha_{s^*}}\circ \alpha_s=\pi\circ\alpha_{s^*}\circ\alpha_s$. Thus  
        \[\beta_{s^*}\circ \beta_s(\pi)=\overline{\overline{\pi\circ\alpha_{s^*}}\circ \alpha_s}=\overline{\pi\circ\alpha_{s^*}\circ\alpha_s}=\overline{\pi|_{A_{ss^*}}}=\pi.\]
        Similarly, $\beta_s\circ \beta_{s^*}=\id$, and hence $\beta_s$ is a bijection with inverse $\beta_{s^*}$.

        To see that $\beta_s$ is continuous, let $h_s\colon U_{s^*s}\to\hat A_{s^*s}$ be the homeomorphism 
        $\pi\mapsto \pi|_{A_{s^*s}}$. Notice that $\beta_s(\pi)=h_{s^*}\inv\big(h_s(\pi)\circ \alpha_{s^*} \big)$. Thus it suffices to show that $\rho\mapsto \rho\circ\alpha_{s^*}$ is continuous from $\hat A_{s^*s}$ to $\hat A_{ss^*}$. Suppose that $\rho_n\to \rho$ in $\hat A_{s^*s}$ and let $V$ be a neighbourhood of $\rho\circ\alpha_{s^*}$.  Say $V=\{\pi\in\hat A_{ss^*}:\pi|_J\neq 0\}$ where $J$ is an ideal in $A_{ss^*}$. Then $\alpha_{s^*}(J)$ is an ideal of $A_{s^*s}$ and $V'=\{\pi\in\hat A_{s^*s}:\pi|_{\alpha_{s^*}(J)}\neq 0\}$ is a neighbourhood of $\rho$. Thus $\rho_n\in V'$ eventually, and then $\rho_n \circ\alpha_{s^*}\in V$ eventually. Thus $\rho\mapsto \rho\circ\alpha_{s^*}$ is continuous, and hence so is $\beta_s$. By symmetry, $\beta_s\inv=\beta_{s^*}$ is also continuous, and hence $\beta_s$ is a homeomorphism.

        To see that $\beta_s\circ\beta_t=\beta_{st}$, we first check that the domains match up.
        Since $\alpha_s\circ\alpha_t=\alpha_{st}$,  we have $\alpha_{t}\inv (A_{tt^*}\cap A_{s^*s})=A_{(st)^*st}$. Thus
       \begin{align*}
            \beta_t^{-1}(U_{tt^*}\cap U_{s^*s})
            &=\{\pi\in\hat{A}:\pi\circ\alpha_{t^*}|_{A_{tt^*}\cap A_{s^*s}}\ne 0\}
             =\{\pi\in\hat{A}:\pi|_{\alpha_{t^*}(A_{tt^*}\cap A_{s^*s})}\ne 0\}\\
             &=\{\pi\in\hat{A}:\pi|_{A_{(st)^*st}}\ne 0\}
             =U_{(st)^*st}.
        \end{align*}
        Finally, \[\beta_s(\beta_t(\pi))=\overline{\beta_t(\pi)\circ\alpha_{s^*}}=\overline{\overline{(\pi\circ\alpha_{t^*})}\circ\alpha_{s^*}}=\overline{\pi\circ\alpha_{t^*s^*}}=\overline{\pi\circ\alpha_{(st)^*}}=\beta_{st}(\pi),\] as needed. 
        Thus $\beta$ is an action. 
        \end{proof}

\begin{notation} Let $I$ be an ideal of a C*-algebra $A$.
Throughout this paper, we 
identify the open subset $U_I\coloneqq \{\pi\in\hat A: \pi|_{I}\ne 0\}$ of $\hat A$ with $\hat I$, and we will write $\hat I$ for $U_I$.  In particular, when  $\alpha\colon S \curvearrowright A$ is an action of a semigroup $S$ on $A$ and $\beta$ is the induced action on $\hat A$ of \Cref{prop: induced action on spectrum}, we write $\beta_s\colon \hat{A}_{s^*s}\to \hat{A}_{ss^*}$ instead of $\beta_s\colon U_{s^*s}\to U_{ss^*}$.
We do this to highlight the connection between the open sets of $\hat{A}$ and the ideals coming from the original action $\alpha$.
\end{notation}

\section{A family of regular representations}\label{sec reg reps}

Let $\alpha\colon S \curvearrowright A$ be an action of an inverse semigroup $S$ on a C*-algebra $A$ and let $\beta\colon S \curvearrowright \hat A$ be the induced action of $S$ on the spectrum $\hat A$ of $A$ from \cref{prop: induced action on spectrum}.   Using $\beta$, we now construct a groupoid of germs $G(\hat{A},S,\beta)$ with unit space homeomorphic to (the possibly non-Hausdorff) $\hat A$. 

Let \[\Omega\coloneqq\{(s,\pi):s\in S, \pi\in \hat{A}_{s^*s}\}.\] We define a relation $\sim$ on $\Omega$ by $(s,\pi)\sim (t,\eta)$ if and only if $\pi=\eta$  and there exists $e\in E(S)$ such that $\pi\in \hat{A}_e$ and $se=te$. It is straightforward to see $\sim$ is an equivalence relation. For $(s,\pi)\in \Omega$, we write $[s,\pi]$ for the equivalence class of $(s,\pi)$; this class is called  a \emph{germ} in the literature.
We will often use the following characterisation in our calculations.

\begin{lemma}\label{lem: equivalence in omega}
\begin{enumerate}
\item Let $(s,\pi), (t, \eta)\in \Omega$. Then $(s,\pi)\sim (t, \eta)$ if and only if $\pi=\eta\in  \hat{I}_{s,t}$. 
\item Let $e,e'\in E(S)$ and $(e,\pi), (e',\pi)\in\Omega$. Then $(e,\pi)\sim (e',\pi)$.
\end{enumerate}
\end{lemma}
\begin{proof}
    First, suppose that $(s,\pi)\sim (t,\eta)$. Then $\pi=\eta$ and there exists $e\in E(S)$ such that $\pi\in \hat{A}_e$ and $se=te$. Then $se\le s,t$ and so $A_{es^*se}\subseteq I_{s,t}$. Since projections commute,  $A_{es^*se}=A_{es^*s}=A_{e}\cap A_{s^*s}$. Thus $\pi\in\hat{A}_{es^*se}\subseteq \hat{I}_{s,t}$.

    Second, suppose that $\pi=\eta$ and that $\pi\in \hat{I}_{s,t}$. 
     Then $I_{s,t}\neq \{0\}$ and $\pi\neq 0$, and so by definition of $I_{s,t}$ there must exist some $v\le s,t$ such that $\pi|_{A_{v^*v}}\neq 0$.  
    Now  $\pi\in \hat{A}_{v^*v}$ and $v=vv^*s=vv^*t$. We have
    \[s(v^*v)=s(s^*vv^*)(vv^*s)=ss^*vv^*s=vv^*ss^*s=vv^*s=v.\] Similarly, $tv^*v=v$. Thus $sv^*v=tv^*v$ and $\pi=\eta$, whence  $(s,\pi)\sim (t,\pi)$.

    Finally, if $(e,\pi)$ and $(e',\pi)\in \Omega$, then $\pi\in \hat A_e\cap \hat A_{e'}=\hat A_{ee'}$ and $e(ee')=e'(ee')$, whence $(e,\pi)\sim(e',\pi)$. 
\end{proof}

  We now consider the  set 
\[G(\hat{A},S,\beta)\coloneqq \Omega/\!\!\sim=\{[s,\pi]:s\in S, \pi\in \hat{A}_{s^*s}\}\]
of germs and equip it with a groupoid structure. Here $\hat A$ may not be Hausdorff, and in the usual constructions, for example in \cite[Theorem~3.3.2]{Paterson-groupoid-book} and \cite[Proposition~4.11]{Exel2008}, the inverse semigroup acts on a locally compact, Hausdorff space. 

To equip $G(\hat{A},S,\beta)$ with a groupoid structure, we define the set of composable pairs by
\[G(\hat{A},S,\beta)^{(2)}\coloneqq\{([s,\pi],[t,\eta])\in G(\hat{A},S,\beta)^{2}:\pi=\beta_t(\eta)\}
,\] 
and multiplication and inversion   by 
\[[s,\pi]\cdot[t,\eta]=[st,\eta]\quad\text{and}\quad [s,\pi]^{-1}=[s^*,\beta_s(\pi)].\]
It is straightforward to show that the operations are well-defined and give a groupoid structure. 
For these algebraic arguments it does not matter that $\hat A$ may not be Hausdorff, so see \cite[Theorem~3.3.2]{Paterson-groupoid-book}  or \cite[Proposition~4.11]{Exel2008} for this. In  \Cref{prop: groupoid of germs} below we carefully check that the usual construction  still gives  a locally compact topology  with respect to which $G(\hat{A},S,\beta)$ is a topological groupoid. The unit space is
 \begin{align*}
    G(\hat{A},S,\beta)^{(0)}&=\{[e,\pi]: e\in E(S), \pi\in \hat{A}_e\}.
 \end{align*}
 Note that  $[e,\pi]=[e',\pi]$ for all $e,e'\in E(S)$ by \Cref{lem: equivalence in omega}. 
 For $\pi \in \hat{A}_e$, we write  \[G_\pi \coloneqq G(\hat{A},S,\beta)_{[e,\pi]}=\{[t,\pi]\colon \pi \in \hat{A}_{t^*t}\}\]
for the fibre over $[e,\pi]$.

We say that a topological space is \emph{locally compact} if every point has a neighbourhood basis consisting of compact neighbourhoods. If $A$ is a C*-algebra, then $\hat A$ is locally compact in this sense by \cite[Corollary~3.3.8]{Dixmier}.

\begin{prop}\label{prop: groupoid of germs} Let $\alpha\colon S \curvearrowright A$ be an action of an inverse semigroup $S$ on a C*-algebra $A$, and let $\beta\colon S \curvearrowright \hat A$ be the induced action of $S$ on the spectrum $\hat A$ of $A$ from \Cref{prop: induced action on spectrum}. 
For $s\in S$ and a subset $U\subseteq \hat{A}_{s^*s}$, let $\Theta(s,U)\coloneqq \{[s,\pi]:\pi\in U\}$.
Then \begin{equation}\label{eq basis}\{\Theta(s,U):s\in S,\ U\text{ is open in } \hat{A}_{s^*s} \}\end{equation} is a basis for a locally compact topology on $G(\hat{A},S,\beta)$. With respect to this topology, $G(\hat{A},S,\beta)$ is a locally compact, \'{e}tale, topological groupoid  with  unit space homeomorphic to $\hat{A}$.
\end{prop}

\begin{proof} As mentioned above, that $G(\hat{A},S,\beta)$ is a groupoid follows as in \cite[Theorem~3.3.2]{Paterson-groupoid-book}  or \cite[Proposition~4.11]{Exel2008} even though $\hat A$ may not be Hausdorff.  Next, we check the topological minutia.

We start by showing that \eqref{eq basis} 
is a basis for a topology.
Since $\beta$ is an action we have $\cup_{s\in S}\hat{A}_{s^*s}=\hat{A}$, and hence $\cup_{s\in S}\Theta(s,\hat{A}_{s^*s})=G(\hat{A},S,\beta)$.
Let  $s,t\in S$, and let $U\subset\hat{A}_{s^*s}$ and $V\subset \hat{A}_{t^*t}$ be open subsets such that $\Theta(s,U)\cap \Theta(t,V)\ne \emptyset$. 
Then there exist $\pi\in U,\eta\in V$ such that $[s,\pi]=[t,\eta]\in \Theta(s,U)\cap \Theta(t,V)$. By \Cref{lem: equivalence in omega},  $\pi=\eta$ and $\pi\in \hat{I}_{s,t}$. Let $W=U\cap V\cap \hat{I}_{s,t}$. Then for all $\xi\in W$ we have $[s,\xi]=[t,\xi]$. Thus $\Theta(s,W)\subseteq \Theta(s,U)\cap \Theta(t,V)$.  Therefore $\{\Theta(s,U):s\in S, U\text{ is open in }  \hat{A}_{s^*s} \}$ is a basis.

To see $G(\hat{A},S,\beta)$ is locally compact in this topology, fix $[s,\pi]\in G(\hat{A},S,\beta)$. Since $\hat{A}$ is locally compact there exists a neighbourhood basis $\BB_\pi$ of $\pi$ consisting of compact neighbourhoods of $\pi$.  Since $\hat{A}_{s^*s}$ is open, we may assume that $B\subseteq \hat{A}_{s^*s}$ for all $B\in \BB_\pi$. 
We claim that \[\{\Theta(s,B):B\in\BB_\pi\}\] forms a neighbourhood basis consisting of compact neighbourhoods of $[s,\pi]$. If $X$ is an open neighbourhood of $[s,\pi]$, then there exists open subset $U$ of $A_{s^*s}$ such that $[s,\pi]\in\Theta(s,U)\subseteq X$. Since $\BB_\pi$ is a neighbourhood basis of $\pi$, there exists a compact neighbourhood $B\in\BB_\pi$ such that $\pi\in B\subseteq U$. Thus $[s,\pi]\in\Theta(s, B)\subseteq \Theta(s,U)\subseteq X$. This shows $\{\Theta(s,B):B\in\BB_\pi\}$ forms a neighbourhood basis; it remains to show that each $\Theta(s,B)$ is  compact.

It suffices to consider an open cover of $\Theta(s,B)$ by a union  of basic sets $\{\Theta(t_i,U_i)\}_{i\in\II}$. Since $B\subseteq \hat A_{s^*s}$ and 
$\Theta(t_i,U_i)\cap \Theta(s,\hat{A}_{s^*s})
=
\Theta(s,U_i\cap \hat{I}_{s,t_i})$ by \Cref{lem: equivalence in omega},  the collection $\{U_i\cap \hat{I}_{s,t_i}:i\in\II\}$ covers $B$. 
Since $B$ is compact there exists a finite $F\subset \II$ such that  $\{U_i\cap \hat{I}_{s,t_i}:i\in F \}$ covers $B$. Then $\{\Theta(t_i,U_i):i\in F \}$ is a finite open subcover of $\Theta(s, B)$. Thus every point in $G(\hat{A},S,\beta)$ has a neighbourhood basis of compact sets and $G(\hat{A},S,\beta)$ is locally compact.

Next we show that multiplication is continuous. Suppose that  \[([s_\lambda,\beta_{t_\lambda}(\eta_\lambda)],[t_\lambda,\eta_\lambda] )\to ([s,\beta_{t}(\eta)],[t,\eta] )\] in $G(\hat{A},S,\beta)^{(2)}$. Fix a basic neighbourhood $\Theta(st, U)$ of $[st,\eta]=[s,\beta_{t}(\eta)][t,\eta]$ with $U\subseteq \hat{A}_{t^*s^*st}$. Then $U\subseteq \hat{A}_{t^*t}$ and $\Theta(t, U)$ is a neighbourhood of $[t, \eta]$. Since  $[t_\lambda,\eta_\lambda]\to [t, \eta]$ in $G(\hat{A},S,\beta)$, we have $[t_\lambda,\eta_\lambda]\in \Theta(t, U)$ eventually, that is, $[t_\lambda, \eta_\lambda]=[t, \eta_\lambda]$ and $\eta_\lambda\in U$ eventually.  
We have $\beta_t(U)\subseteq\beta_t(\hat{A}_{(st)^*st})\subseteq \hat{A}_{s^*s}$, and so 
$[s_\lambda,\beta_{t_\lambda}(\eta_\lambda)]=[s_\lambda,\beta_{t}(\eta_\lambda)]\in \Theta(s,\beta_t(U))$ 
eventually, 
that is, $[s_\lambda, \beta_t(\eta_\lambda)]=[s, \beta_t(\eta_\lambda)]$ eventually. 
Thus eventually we have
\[
[s_\lambda, \beta_{t_\lambda}(\eta_\lambda)][t_\lambda, \eta_\lambda]=[s, \beta_t(\eta_\lambda)][t, \eta_\lambda]=[st,\eta_\lambda]\in \Theta(st, U),
\]
as needed.
Thus multiplication is continuous.

To show that inversion is continuous, suppose that $[s_\lambda,\pi_\lambda] \to [s,\pi]$. Let $\Theta(s^*,U)$ be a basic neighbourhood of $[s,\pi]^{-1}=[s^*,\beta_{s}(\pi)]$; we may assume that   $U\subseteq \hat{A}_{ss^*}$.  Then $[s,\pi]\in \Theta(s,\beta_{s^*}(U))$. Thus $[s_\lambda,\pi_\lambda]$ is eventually in $\Theta(s,\beta_{s^*}(U))$, that is,  $[s_\lambda,\pi_\lambda]=[s, \pi_\lambda]$ and $\pi_\lambda\in \beta_{s^*}(U)$ eventually. Thus eventually $\beta_{s_\lambda}(\pi_\lambda)=\beta_s(\pi_\lambda)\in \beta_s(\beta_{s^*}(U))=U$ and $[s_\lambda,\pi_\lambda]^{-1}=[s ,\pi_\lambda]^{-1}=[s^*,\beta_{s}(\pi_\lambda)]\in \Theta(s^*,U)$. Now inversion is continuous.  Thus $G(\hat{A},S,\beta)$ is a topological groupoid.

Since multiplication and inversion are continuous, the range map $r\colon x\mapsto xx^{-1}$ is continuous. 
To see that $r$ is a local homeomorphism, it suffices to show that the unit space is open and $r$ is open (see \cite[Lemma~1.25]{Williams-groupoid-book}, the proof of which does not use that the groupoid or its unit space is Hausdorff).  Since $\hat{A}=\cup_{s\in S}\hat{A}_{s^*s}$ we see that $G(\hat{A},S,\beta)^{(0)}=\cup_{s\in S}\Theta(s,\hat{A}_{s^*s})$ is open.
Furthermore, $r(\Theta(s, U))=\Theta(ss^*,\beta_s(U))$. Since $\{\Theta(s,U):s\in S, U\text{ is open in } \hat{A}_{s^*s} \}$ is a basis for the topology, it follows that $r$ is open. Thus $G(\hat{A},S,\beta)$ is \'etale. 
\end{proof}

\begin{remark}
   We note that $G(\hat{A},S,\beta)$ is isomorphic to, for example,  the ``dual groupoid'' $\hat{A}\rtimes S$ from \cite[Definition~2.14]{KwasniewskiMeyer2021} which is used in the characterisation of simplicity of the essential crossed product $A\rtimes_\ess S$ in \cite[Corollary~6.15]{KwasniewskiMeyer2021}.  A related groupoid is the transformation groupoid  $\Prim(A)\rtimes S$ of \cite[Theorem~6.5]{BussExelMeyer2017} which detects when a certain weak conditional expectation  $E\colon A\rtimes_\red S\to A''$ takes values in $A$.
\end{remark}

Let $\pi\in\hat A$, and write $\pi\colon A\to B(H_\pi)$. 
We now build a covariant representation $(\rho_\pi, \sigma_\pi)$  on $\ell^2(G_\pi,H_\pi)$, 
the Hilbert space of square-integrable functions on the fibre $G_\pi$ of 
$G(S,\beta,\hat A)$ with values in $H_\pi$.

\begin{notation}\label{notation: point mass}
Let $\pi\in \hat{A}$ and let $h\in H_\pi$. We write $\chi^h_{[s,\pi]}$ for the function in $\ell^2(G_\pi,H_\pi)$ such that $\chi^h_{[s,\pi]}([t,\pi])=h$ if $[s,\pi]=[t,\pi]$ and is $0$ otherwise. 
\end{notation}

\begin{thm}\label{thm: rep induced from hat A}
Let $\alpha\colon S \curvearrowright A$ be an action of an inverse semigroup $S$ on a C*-algebra $A$, and let $\beta\colon S \curvearrowright \hat A$ be the induced action of $S$ on the spectrum $\hat A$ of $A$ from \Cref{prop: induced action on spectrum}.  Let $\pi\colon A\to B(H_\pi)$ be an  irreducible representation. Let $\xi\in\ell^2(G_\pi, H_\pi)$ and define 
$\rho_\pi\colon A\to B(\ell^2(G_\pi, H_\pi))$ and $\sigma_\pi\colon S\to B(\ell^2(G_\pi, H_\pi))$ by 
\begin{align*}
(\rho_\pi(a)\xi)([t,\pi])&=\beta_t(\pi)(a)\big(\xi([t,\pi])\big)\text{\ and }\\
(\sigma_\pi(s)\xi)([t,\pi])&=\begin{cases}
\xi([s^*t,\pi])&\text{if $\beta_t(\pi) \in \hat{A}_{ss^*}$ }\\ 0&\text{otherwise.}
\end{cases}
\end{align*}
Then  $(\rho_\pi,\sigma_\pi)$ is a covariant representation of $\alpha\colon S \curvearrowright A$.
\end{thm}
\begin{proof}
If $\pi\in \hat A_{t^*t}$ and $\beta_t(\pi)\in \hat A_{ss^*}$, then $[s^*, \beta_t(\pi)]$ and $[t,\pi]$ are composable in the groupoid of \Cref{prop: groupoid of germs} with product $[s^*t,\pi]$; thus $\pi\in\hat A_{(s^*t)^*(s^*t)}$. It is straightforward to check that $\rho_\pi(a)\xi$ and $\sigma_\pi(s)\xi$ are in $\ell^2(G_\pi, H_\pi)$.

We check  that $\rho_\pi$ is a non-degenerate representation. Let $a, b\in A$ and $\xi, \zeta\in \ell^2(G_\pi, H_\pi)$. Since $\beta_t(\pi)$ is a homomorphism we have
\begin{align*}
(\rho_\pi(ab)\xi)([t,\pi])
&=\beta_t(\pi)(ab)\big(\xi([t,\pi])\big)
=\big(\beta_t(\pi)(a)\beta_t(\pi)(b)\big)\big(\xi([t,\pi])\big)\\
&=\beta_t(\pi)(a)\big((\rho_\pi(b)\xi)([t,\pi]) \big)
=\rho_\pi(a)\big(\rho_\pi(b)\xi \big) ([t,\pi])\\
&=((\rho_\pi(a) \rho_\pi(b))\xi) ([t,\pi]),
\end{align*}
whence $\rho_\pi(ab)=\rho_\pi(a) \rho_\pi(b)$. Further, 
\begin{align*}
    (\rho_\pi(a)\xi\mid\zeta)
    &=\sum_{[t,\pi]\in G_\pi} \big(\beta_t(\pi)(a)(\xi([t,\pi]))\mid \zeta([t,\pi]) \big)\\
    &=\sum_{[t,\pi]\in G_\pi} \big(\xi([t,\pi])\mid \beta_t(\pi)(a)^*(\zeta([t,\pi])) \big)\\
     &=\sum_{[t,\pi]\in G_\pi} \big(\xi([t,\pi])\mid \beta_t(\pi)(a^*)(\zeta([t,\pi])) \big)\\
     &=(\xi\mid \rho_\pi(a^*)\zeta),
\end{align*}
whence $\rho_\pi(a)^*=\rho_\pi(a^*)$.

To see that $\rho_\pi$ is non-degenerate, let  $[s,\pi]\in G_\pi$ and $h\in H_\pi$. It suffices to approximate the function $\chi_{[s,\pi]}^h$ of \Cref{notation: point mass}  by functions in $\vecspan\{\rho_\pi(a)\xi:a\in A, \xi\in \ell^2(G_\pi,H_\pi)\}$.  
Since  $[s,\pi]\in G_\pi$, we have  $0\neq\beta_s(\pi)\in \hat A$ and hence $\beta_s(\pi)$ is non-degenerate. Let $\{i_\lambda\}$ be an approximate identity for $A$; then $\beta_s(\pi)(i_\lambda)h\to h$.
Now
\begin{align*}
\|\rho_\pi(i_\lambda)\chi_{[s,\pi]}^{h}- \chi_{[s,\pi]}^h \|^2
&=\sum_{[t,\pi]\in G_\pi} \|\rho_\pi(i_\lambda)\chi_{[s,\pi]}^{h})([t,\pi])- \chi_{[s,\pi]}^h([t,\pi])\|^2\\
&=\sum_{[t,\pi]=[s,\pi]}\|\beta_{t}(\pi)(i_\lambda)(\chi_{[s,\pi]}^{h}([t,h]))-\chi_{[s,\pi]}^h([t,\pi])\|^2\\
&=\|\beta_{s}(\pi)(i_\lambda)h-h\|^2\to 0.
\end{align*}
Thus $\rho_\pi$ is non-degenerate. 

We will now check the three items in \Cref{defn covariant pair}. 
Fix $r,s,t \in S$ and $\xi, \zeta \in \ell^2(G_\pi, H_\pi)$.
To see that $\sigma_\pi(r)\sigma_\pi(s)=\sigma_\pi(rs)$ we compute
\begin{align*}
(\sigma_\pi(r)(\sigma_\pi(s)\xi))([t,\pi])&=\begin{cases}
(\sigma_\pi(s)\xi)([r^*t,\pi])&\text{if $\beta_t(\pi) \in \hat{A}_{rr^*}$ }\\ 0&\text{otherwise}
\end{cases}\\
&=\begin{cases}
\xi([s^*r^*t,\pi])&\text{if $\beta_t(\pi) \in \hat{A}_{rr^*}$ and $\beta_{r^*t}(\pi) \in \hat{A}_{ss^*}$ }\\ 0&\text{otherwise}
\end{cases}
\\
(\sigma_\pi(rs)\xi)([t,\pi])&=\begin{cases}
\xi([s^*r^*t,\pi])&\text{if $\beta_t(\pi) \in \hat{A}_{rss^*r^*}$ }\\ 0&\text{otherwise.}
\end{cases}
\end{align*}
If $\beta_t(\pi) \in \hat{A}_{rr^*}$ and $\beta_{r^*t}(\pi) \in \hat{A}_{ss^*}$, then \[\beta_t(\pi)=\beta_r(\beta_{r^*t}(\pi)) \in \hat{A}_{rss^*r^*}.\]
If   $\beta_t(\pi) \in \hat{A}_{rss^*r^*}$, then  
\[\beta_{r^*t}(\pi)=\beta_{r^*}(\beta_t(\pi)) \in \hat{A}_{r^*rss^*rr^*}=\hat{A}_{r^*rss^*}\subseteq \hat{A}_{ss^*}.\]
Thus $\beta_t(\pi) \in \hat{A}_{rr^*}$ and $\beta_{r^*t}(\pi) \in \hat{A}_{ss^*}$ if and only if $\beta_t(\pi) \in \hat{A}_{rss^*r^*}$, and a glance at the formulas above gives
 $\sigma_\pi(r)\sigma_\pi(s)=\sigma_\pi(rs)$.

To see that $\sigma_\pi(s^*)=\sigma_\pi(s)^*$ we compute
\begin{align}
(\sigma_\pi(s)\xi\mid\zeta)&=\sum_{[t,\pi] \in G_\pi}(\sigma_\pi(s)(\xi)([t,\pi])\mid\zeta([t,\pi]))_{H_\pi}\notag\\
&=\sum_{\{[t,\pi] \in G_\pi: \beta_t(\pi) \in \hat{A}_{ss^*}\}}(\xi([s^*t,\pi])\mid\zeta([t,\pi]))_{H_\pi};\label{eq: sigma_pis} 
\\
(\xi\mid\sigma_\pi(s^*)\zeta)&=\sum_{[u,\pi] \in G_\pi}(\xi([u,\pi])\mid\sigma_\pi(s^*)(\zeta)([u,\pi]))_{H_\pi}\notag\\
&=\sum_{\{[u,\pi] \in G_\pi: \beta_u(\pi) \in \hat{A}_{s^*s}\}}(\xi([u,\pi])\mid\zeta([su,\pi]))_{H_\pi}.\label{eq: sigma_pis*}
\end{align}
To reconcile \Cref{eq: sigma_pis} and  \Cref{eq: sigma_pis*}, first let $t\in S$ such that $[t,\pi] \in G_\pi$ and $\beta_t(\pi) \in \hat{A}_{ss^*}$. By definition of $[t,\pi]$, we have $\pi\in\hat A_{t^*t}$. Let $u=s^*t$. Then $\pi\in \hat A_{u^*u}=\hat A_{t^*ss^*t}$ and $\beta_u(\pi)=\beta_{s^*t}(\pi)\in\hat A_{s^*tt^*s}\subseteq \hat A_{s^*s}$. To see  that $(su,\pi)\sim (t,\pi)$ we note that $su(u^*u)=ss^*t=t(u^*u)$ and $\pi\in \hat A_{u^*u}$; thus $[t,\pi]=[su,\pi]$.  This shows every summand in \Cref{eq: sigma_pis} equals a summand in \Cref{eq: sigma_pis*}.  Second, let $u\in S$ such that
$[u,\pi] \in G_\pi$ and $\beta_u(\pi) \in \hat{A}_{s^*s}$. Thus $\pi\in\hat A_{u^*u}$. Let $t=su$. Then $[t,\pi]\in G_\pi$ since $[su,\pi]\in G_\pi$. We have $\beta_t(\pi)=\beta_{su}(\pi)\in 
\hat A_{suu^*s}\subseteq \hat A_{ss^*}$. To see that $(u,\pi)\sim (s^*t,\pi)$ we note that $u(t^*t)=s^*su=s^*t(t^*t)$ and $\pi\in \hat A_{t^*t}$; thus $[u,\pi]=[s^*t,\pi]$.  This shows every summand in \Cref{eq: sigma_pis*} equals a summand in \Cref{eq: sigma_pis}, and we conclude that $\sigma_\pi(s^*)=\sigma_\pi(s)^*$.

Next we check the covariance. Let $a\in A_{s^*s}$.  Then
\begin{align*}
(\sigma_\pi(s)(\rho_\pi(a)(\sigma_\pi(s^*)\xi)))([t,\pi])&=\begin{cases}
(\rho_\pi(a)(\sigma_\pi(s^*)\xi))([s^*t,\pi]) &\text{if $\beta_{t}(\pi) \in \hat{A}_{ss^*}$}\\
0&\text{otherwise}
\end{cases}\\
&=\begin{cases}
\beta_{s^*t}(\pi)(a)(\sigma_\pi(s^*)\xi)([s^*t,\pi]) &\text{if $\beta_{t}(\pi) \in \hat{A}_{ss^*}$}\\
0&\text{otherwise}
\end{cases}\\
&=\begin{cases}
\beta_{s^*t}(\pi)(a)\xi([ss^*t,\pi]) &\text{if $\beta_{t}(\pi) \in \hat{A}_{ss^*}$ and $\beta_{s^*t}(\pi) \in \hat{A}_{s^*s}$}\\
0&\text{otherwise.}
\end{cases}
\end{align*}
If $\beta_{t}(\pi) \in \hat{A}_{ss^*}$, then $\beta_{s^*t}(\pi) \in \hat{A}_{s^*s}$. Further, $ss^*t(t^*ss^*t)=t(t^*ss^*t)$ and $\pi\in \hat A_{(s^*t)^*s^*t}$, giving $[ss^*t,\pi]=[t,\pi]$.
Thus 
\begin{align*}
(\sigma_\pi(s)(\rho_\pi(a)(\sigma_\pi(s^*)\xi)))([t,\pi])&=
\begin{cases}
\beta_{s^*t}(\pi)(a)\xi([t,\pi]) &\text{if $\beta_{t}(\pi) \in \hat{A}_{ss^*}$ }\\
0&\text{otherwise.}
\end{cases}
\end{align*}
On the other hand, if $\beta_t(\pi)\notin\hat A_{ss^*}$, then $\beta_t(\pi)(\alpha_s(a))=0$, whence 
\begin{align*}
\rho_\pi(\alpha_s(a))\xi([t,\pi])&=\beta_{t}(\pi)(\alpha_s(a))\xi([t,\pi])\\
&=\begin{cases}
\beta_{t}(\pi)(\alpha_s(a))\xi([t,\pi]) &\text{if $\beta_{t}(\pi) \in \hat{A}_{ss^*}$ }\\
0&\text{otherwise}
\end{cases}\\
&=\begin{cases}
\beta_{s^*t}(\pi)(a)\xi([t,\pi]) &\text{if $\beta_{t}(\pi) \in \hat{A}_{ss^*}$ }\\
0&\text{otherwise.}
\end{cases}
\end{align*}
Thus $\sigma_\pi(s)\rho_\pi(a)\sigma_\pi(s^*)=\rho_\pi(\alpha_s(a))$.

Finally, fix $e \in E(S)$; we will show that $\rho_\pi(A_e) \ell^2(G_\pi, H_\pi)=\sigma_\pi(e)\ell^2(G_\pi, H_\pi)$. Observe  that  if $\beta_t(\pi) \in \hat{A}_e$ then $\beta_t(\pi)=\beta_{et}(\pi)$ in which case $[t,\pi]=[et,\pi]$. Thus 
\begin{equation}\label{eq: ell2 functions}
    \sigma_\pi(e)\ell^2(G_\pi, H_\pi)=\{\xi \in \ell^2(G_\pi, H_\pi):\xi([t,\pi])=0 \text{ for all } t \in S \text{ such that } \beta_t(\pi)\not \in \hat{A}_e\}. 
\end{equation}
Let  $a \in A_e$ and $\xi \in \ell^2(G_\pi, H_\pi)$. 
Then
\begin{align*} \rho_\pi(a)\xi([t,\pi])
&=\beta_t(\pi)(a)\xi([t,\pi])
=\begin{cases}
\beta_t(\pi)(a)\xi([t,\pi])&\text{if $\beta_t(\pi) \in \hat{A}_e$}\\0&\text{otherwise,}
\end{cases}
\end{align*}
and \Cref{{eq: ell2 functions}} gives $\rho_\pi(a)\xi \in \sigma_\pi(e)\ell^2(G_\pi, H_\pi)$.
Thus $\rho_\pi(A_e) \ell^2(G_\pi, H_\pi)\subseteq \sigma_\pi(e)\ell^2(G_\pi, H_\pi)$.

For the reverse containment, it suffices to show that $\sigma_\pi(e)\xi\in \rho_\pi(A_e) \ell^2(G_\pi, H_\pi)$ for  $\xi\in \ell^2(G_\pi,H_\pi)$ with finite support $F$.  Note that $F\subseteq \{[t,\pi]:\beta_t(\pi)\in \hat A_e\}$ using  \Cref{eq: ell2 functions}.
Let $\{i_\lambda\}$ be an approximate identity for $A_e$.  For $[t,\pi]\in F$ we have $\pi\in\hat A$ and $\beta_t(\pi)|_{A_e}\neq 0$; thus $\beta_t(\pi)|_{A_e}$ is non-degenerate and $\beta_t(\pi)(i_\lambda)\to \id_{H_\pi}$ strongly. 
Now
\begin{align*}
\rho_\pi(i_\lambda)\xi([t,\pi])
&= \begin{cases}
\beta_t(\pi)(i_\lambda)\xi([t,\pi])&\text{if $\beta_t(\pi) \in \hat{A}_e$}\\
0&\text{otherwise}
\end{cases}\\
&\to\begin{cases}
\xi([t,\pi])&\text{if $\beta_t(\pi) \in \hat{A}_e$}\\
0&\text{otherwise}
\end{cases}\\
&=\xi([t,\pi]).
\end{align*}
Fix $\epsilon>0$. Since $F$ is finite there exists $\lambda_0$ such that
\[
\lambda>\lambda_0\Longrightarrow \|\beta_t(\pi)(i_\lambda)\xi([t,\pi])-\xi([t,\pi])\|^2<\frac{\epsilon^2}{|F|}\text{ for all $[t,\pi]\in F$.}
\]
Now for $\lambda>\lambda_0$ we have
\begin{align*}\|\rho_\pi(i_\lambda)\xi-\xi\|^2
&=\sum_{[t,\pi]\in G_\pi}\|\rho_\pi(i_\lambda)\xi([t,\pi])-\xi([t,\pi])\|^2\\
&=\sum_{[t,\pi]\in F}\|\rho_\pi(i_\lambda)\xi([t,\pi])-\xi([t,\pi])\|^2<\epsilon^2.
\end{align*}
 Since $\rho_\pi(A_e) \ell^2(G_\pi, H_\pi)$ is closed by \cite[Theorem~32.22 on page~268]{Hewitt-Ross-2}, we see  \[\rho_\pi(A_e) \ell^2(G_\pi, H_\pi)=\sigma_\pi(e)\ell^2(G_\pi, H_\pi).\qedhere\]
\end{proof}

\begin{prop}\label{prop: covariant reps are unitarily equivalent}
Let $\alpha\colon S \curvearrowright A$ be an action of an inverse semigroup $S$ on a C*-algebra $A$, and let $\beta\colon S \curvearrowright \hat A$ be the induced action of $S$ on the spectrum $\hat A$ of $A$ from \Cref{prop: induced action on spectrum}.
Let $[s, \pi]\in G_\pi$. For $\xi\in \ell^2(G_\pi,H_\pi)$ and $[t, \beta_s(\pi)]\in G_{\beta_s(\pi)}$ define
    \[U_{[s,\pi]}\xi([t,\beta_s(\pi)])
    =\xi([ts,\pi]).\] 
For $[r,\pi]\in G_\pi$ and $h\in H_\pi$, let  $\chi_{[r,\pi]}^h\in\ell^2(G_\pi,H_\pi)$ be the point mass of \Cref{notation: point mass}. Then $U_{[s,\pi]}(\chi^h_{[r,\pi]})=\chi^h_{[rs^*,\beta_s(\pi)]}$, and  \[U_{[s,\pi]}\colon \ell^2(G_\pi,H_\pi)\to \ell^2(G_{\beta_s(\pi)},H_{\pi})\] 
is a unitary operator 
that implements a unitary equivalence of $\rho_\pi\times\sigma_\pi$ and $\rho_{\beta_s(\pi)}\times \sigma_{\beta_s(\pi)}$.
\end{prop}

\begin{proof} To see that $U_{[s,\pi]}$ is well-defined, we observe that 
$[t,\beta_s(\pi)]$ and $[s,\pi]$  are composable  in the groupoid $G(S,\beta,\hat A)$ of germs,  with $[t,\beta_s(\pi)][s,\pi]=[ts,\pi]\in G_\pi$;  in particular, $[ts,\pi]\in G_\pi$ as needed. Furthermore,  $[t,\beta_s(\pi)]\mapsto [t,\beta_s(\pi)][s,\pi]=[ts,\pi]$  is a bijection of $G_{\beta_s(\pi)}$ onto $G_\pi$.
Thus
\begin{align*}
\|U_{[s,\pi]}\xi\|^2
&=\sum_{[t,\beta_s(\pi)]\in G_{\beta_s(\pi)}}\|U_{[s,\pi]}\xi([t,\beta_s(\pi)])\|^2\\
&=\sum_{[t,\beta_s(\pi)]\in G_{\beta_s(\pi)}}\|\xi([ts,\pi])\|^2\\
&
    =\sum_{[u,\pi]\in G_{\pi}}\|\xi([u,\pi])\|^2
    =\|\xi\|^2,
\end{align*}
showing both that $U_{[s,\pi]}$ is well-defined and is an isometry.

We have  
\begin{align*}
    U_{[s,\pi]}(\chi^h_{[r,\pi]})([t,\beta_s(\pi)])
    &=\chi^h_{[r,\pi]}([ts,\pi])
= \begin{cases}
    h&\text{if $[r,\pi]=[t,\beta_s(\pi)][s,\pi]$}\\0&\text{otherwise}
\end{cases}\\
&=\chi^h_{[r,\pi][s,\pi]^{-1}}([t,\beta_s(\pi)])
=\chi^h_{[rs^*,\beta_s(\pi)]}([t,\beta_s(\pi)]).
\end{align*}  
Thus for every $h\in H_\pi$ and $[t,\beta_s(\pi)]$ we have   $U_{[s,\pi]}(\chi^h_{[ts,\pi]})=\chi^h_{[t,\beta_s(\pi)]}$.  Since the functions of the form $\chi^h_{[t,\beta_s(\pi)]}$ span a dense subset of $\ell^2(G_{\beta_s(\pi)}, H_\pi)$ it follows that  $U_{[s,\pi]}$ has dense range and is a unitary.

To see that $U_{[s,\pi]}$ implements the unitary equivalence of the integrated forms $\rho_\pi\times\sigma_\pi$ and $\rho_{\beta_s(\pi)}\times\sigma_{\beta_s(\pi)}$, it suffices to see that $U_{[s,\pi]}$ implements unitary equivalences of  $\rho_\pi$ and $\rho_{\beta_s(\pi)}$, and $\sigma_\pi$ and $\sigma_{\beta_s(\pi)}$.  
Let $a\in A$, $r\in S$ and compute for   $[t, \beta_s(\pi)]\in G_{\beta_s(\pi)}$:
\begin{align*}
U_{[s,\pi]}(\rho_\pi(a)\xi)([t, \beta_s(\pi)])
&=(\rho_\pi(a)\xi)([ts, \pi])\\
&=\beta_{ts}(\pi)(a)(\xi([ts, \pi]))\\
&=\beta_{t}(\beta_s(\pi))(a)\big(U_{[s,\pi]}\xi\big)([t, \beta_s(\pi)])\\
&=\rho_{\beta_s(\pi)}(a)\big( U_{[s,\pi]}\xi \big)([t, \beta_s(\pi)])
\end{align*}
and 
\begin{align*}
U_{[s,\pi]}(\sigma_\pi(r)\xi)([t,\beta_s(\pi)])
&=(\sigma_\pi(r)\xi)([ts,\pi])\\
&=
\begin{cases}
    \xi([r^*ts,\pi])&\text{if $\beta_{ts}(\pi)\in \hat A_{rr^*}$}\\
    0&\text{else}
\end{cases}\\
&=\begin{cases}
    (U_{[s,\pi]}\xi)([r^*t,\beta_s(\pi)])&\text{if $\beta_{t}(\beta_{s}(\pi))\in \hat A_{rr^*}$}\\
    0&\text{else}
\end{cases}\\
&=\sigma_{\beta_s(\pi)}(r)(U_{[s,\pi]}\xi)([t,\beta_s(\pi)]).
\end{align*}
Thus  $\rho_\pi\times\sigma_\pi$ and $\rho_{\beta_s(\pi)}\times \sigma_{\beta_s(\pi)}$ are unitarily equivalent. 
\end{proof}

We now show that taking the quotient of $A\rtimes_\alg S$ by the intersection of the kernels of the regular representations $\{\rho_\pi\times\sigma_\pi:\pi\in\hat A\}$ gives a norm on $A\rtimes_\alg S $. Parts of the following proof are adapted from \cite[Proposition~4.3]{BussExelMeyer2017}.

\begin{prop}\label{prop: reduced norm is a norm}
Let $\alpha\colon S \curvearrowright A$ be an action of an inverse semigroup $S$ on a C*-algebra $A$.
For $\pi\in \hat{A}$,  let $(\rho_\pi,\sigma_\pi)$ be the covariant representation of \Cref{thm: rep induced from hat A}.
For $a\in C_c(S,A)$ set 
\[\|a\|_\red\coloneqq\sup_{\pi\in \hat{A}}\|\rho_\pi\times \sigma_\pi(a)\|.\]
Then $\|\cdot\|_\red$ is a seminorm on $C_c(S,A)$ with kernel $\II_\alpha=\vecspan\{a\delta_r-a\delta_s: r,s\in S, a\in I_{s^*,r^*}\}$ and descends to a norm, also denoted $\|\cdot\|_\red$,  on $A\rtimes_\alg S$.
\end{prop}

\begin{proof}
Temporarily set 
\[
\II\coloneqq\{b\in C_c(S,A): \rho_\pi\times\sigma_\pi(b)=0 \text{ for all $\pi\in\hat A$.}\}
\]
We have $\II_\alpha\subseteq \II$ by \Cref{lem: I_alpha and cov reps}.  
Every $a\in C_c(S,A)$ can be written $a=\sum_{s\in F_a}a_s\delta_s$ where $F_a$ is a finite subset of $S$ and $a_s\in A_{ss^*}$. We will show that $\II\subseteq\II_\alpha$ by induction on $|F_a|$ for $a\in \II$.

Let $a\in \II$ with $|F_a|=1$, that is, $a=a_r\delta_r$ and $F_a=\{r\}$. 
Fix $\pi\in \hat{A}_{rr^*}$, fix $h\in H_\pi$,  
and let $\chi_{[r^*,\pi]}^h\in\ell^2(G_\pi,H_\pi)$ be the point mass of \Cref{notation: point mass}.
For $e\in E(S)$ we have $\beta_e=\id$ and $[r^*e,\pi]=[r^*,\pi]$, and then
\[
0=\big(\rho_\pi\times\sigma_\pi(a_r\delta_r)\chi_{[r^*,\pi]}^h\big)([e,\pi])
=\beta_e(\pi)(a_r)\chi_{[r^*,\pi]}^h)([r^*e,\pi])
=\pi(a_r)h.
\]
Thus $\pi(a_r)h=0$ for all $\pi\in \hat{A}_{rr^*}$ and $h\in H_\pi$, and hence $a_r=0$. Thus $a_r\delta_r=0\in \II_\alpha$.

Let $n\geq 1$ and assume that if $b=\sum_{s\in F_b}b_s\delta_s\in \II$ with $|F_b|\leq n$, then $b\in \II_\alpha$. 
Let $a=\sum_{s\in F_a}a_s\delta_s\in\II$ with $|F_a|=n+1$. 
 Pick $r\in F_a$.
There are  two cases to consider: $\hat{A}_{rr^*}=\cup_{s\in F_a\setminus\{r\}}\hat{I}_{s^*,r^*}$ and $\hat{A}_{rr^*}\setminus\cup_{s\in F_a\setminus\{r\}}\hat{I}_{s^*,r^*}\neq \emptyset$ .

   First, suppose that $\hat{A}_{rr^*}=\cup_{s\in F_a\setminus\{r\}}\hat{I}_{s^*,r^*}$. Thus $A_{rr^*}=\overline{\sum_{s\in F_a\setminus\{r\}}I_{s^*,r^*}}$. Since the span of finitely many closed ideals is again a closed ideal by \cite[Theorem~3.1.7]{Murphy1990} we have $A_{rr^*}=\sum_{s\in F_a\setminus\{r\}}I_{s^*,r^*}$. Now there exist $c_s\in I_{s^*,r^*}$ such that 
    \begin{equation}\label{eq: back to here}
    a_r=\sum_{s\in F_a\setminus\{r\}}c_s.
    \end{equation}
We have 
\[a_r\delta_r-\sum_{s\in F_a\setminus\{r\}}c_s\delta_s=\sum_{s\in F_a\setminus\{r\}}c_s\delta_r-c_s\delta_s\in \II_\alpha.\] 
Since $\II_\alpha\subseteq \II$, we have 
 $a-\sum_{s\in F_a\setminus\{r\}}(c_s\delta_r-c_s\delta_s)\in \II$.  Moreover, 
    \begin{align*}
        a-\sum_{s\in F_a\setminus\{r\}}(c_s\delta_r-c_s\delta_s)
        &=\sum_{s\in F_a}a_s\delta_s-\sum_{s\in F_a\setminus\{r\}}(c_s\delta_r-c_s\delta_s)\\
        &=\sum_{s\in F_a\setminus\{r\}}(a_s+c_s)\delta_s
        +a_r\delta_r-\Big(\sum_{s\in F_a \setminus\{r\}}c_s\Big)\delta_r\\
        &=\sum_{s\in F_a\setminus\{r\}}(a_s+c_s)\delta_s.
    \end{align*}
    Now, by the  induction hypothesis, $a-\sum_{s\in F_a\setminus\{r\}}c_s\delta_r-c_s\delta_s\in\II_\alpha$, and then so  is $a$.

Second, suppose that $\hat{A}_{rr^*}\setminus\cup_{s\in F_a\setminus\{r\}}\hat{I}_{s^*,r^*}\neq \emptyset$.  Fix $\pi \in \hat{A}_{rr^*}\setminus\cup_{s\in F_a\setminus\{r\}}\hat{I}_{s^*,r^*}$ and $h\in H_\pi$. For $e\in E(S)$ we have 
\[
        0=(\rho_\pi\times\sigma_\pi(a)\chi_{[r^*,\pi]}^h)([e,\pi])
        =\sum_{\stackrel{s\in F_a}{[r^*,\pi]=[s^*e,\pi]}}\beta_e(\pi)(a_s)h
       =\pi(a_r)h.
\]
because $[r^*,\pi]=[s^*e,\pi]$ if and only if $\pi\in \hat I_{s^*, r^*}$ by \Cref{lem: equivalence in omega}.   
We conclude that $\pi(a_r)=0$ for all $\pi \in \hat{A}_{rr^*}\setminus\cup_{s\in F_a\setminus\{r\}}\hat{I}_{s^*,r^*}$.  Thus $a_r\in \langle I_{s^*,r^*}: s\in F_a\setminus\{r\}\rangle$ so that $a_r=\sum_{s\in F_a\setminus\{r\}}c_s$ for some $c_s\in I_{s^*, r^*}$ as at \Cref{eq: back to here}. Now repeat the argument above to get $a\in \II_\alpha$ as needed. 
\end{proof}

\begin{cor} Let $\alpha\colon S \curvearrowright A$ be an action of an inverse semigroup $S$ on a C*-algebra $A$. Then
    \[\II_\alpha=\cap\{\ker(\rho\rtimes\sigma): (\rho,\sigma) \text{ is a covariant representation}\}.\]
\end{cor}

\begin{proof}
    Using first \Cref{lem: I_alpha and cov reps} and then \Cref{prop: reduced norm is a norm} we have
\begin{align*}
     \II_\alpha
     &
     \subseteq \cap\{\ker(\rho\rtimes\sigma): (\rho,\sigma) \text{ is a covariant representation}\}
     \\
     &\subseteq \cap\{\ker(\rho_\pi\times\sigma_\pi): \pi\in\hat A\}=\II_\alpha.\qedhere    
\end{align*}
\end{proof}

Let $X$ be a topological space and $V\subseteq X$. We say that  $V$  is  \emph{nowhere dense} if its closure has empty interior and say $V$ is \emph{meagre} if it is the countable union of nowhere dense subsets.  Also,  a subset  is \emph{comeagre} if its complement is meagre.

Let $a\in A\rtimes_\alg S$.
Motivated by \cite[Definition~4.9]{KwasniewskiMeyer2021} we define 
\begin{equation}\label{eq ess norm}\|a\|_\ess\coloneqq \inf\Big\{\sup_{\pi\in V}\|\rho_\pi\times \sigma_\pi(a)\|\colon  
V\subseteq \hat{A}  \text{ is comeagre}\Big\}.
\end{equation}
In general,  $\|\cdot\|_\ess$ is only a seminorm on $A\rtimes_\alg S$; for example \cite[Example~2.25]{ACaHMT} shows that $\|\cdot\|_\red$ and $\|\cdot \|_{\ess}$ may not coincide. We set $\mathcal{N}_{\ess}\coloneqq\{a\in A\rtimes_\alg S:\|a\|_{\ess}=0\}$. 

\begin{defn}\label{defn our crossed products}
    Let $\alpha\colon S \curvearrowright A$ be an action of an  inverse semigroup $S$ on a C*-algebra $A$. Define $\Red(A,S,\alpha)$ to be the C*-algebraic (hence Hausdorff) completion of $A\rtimes_\alg S$ with respect to the norm $\|\cdot\|_r$ of \Cref{prop: reduced norm is a norm} and define $\Ess(A,S,\alpha)$  to be the C*-algebraic completion of $(A\rtimes_\alg S)/\mathcal{N}_\ess$ with respect to the seminorm $\|\cdot\|_\ess$ of \Cref{eq ess norm}. 
\end{defn}
Since $\|a\|_r\ge\|a\|_\ess$ we see that $\Ess(A,S,\alpha)$ is  a quotient of $\Red(A,S,\alpha)$.
In the next section we will prove that $\Red(A,S,\alpha)$  is isomorphic to the reduced   crossed product in \cite{BussExelMeyer2017}, and,  if $S$ is quasi-countable, then $\Ess(A,S,\alpha)$ is isomorphic to the essential crossed product in \cite{KwasniewskiMeyer2021}.

\section{Reconciliation of the C*-algebras}\label{sec reconcile}
Let $\alpha\colon S \curvearrowright A$ be an action of an  inverse semigroup $S$ on a C*-algebra $A$. 
In this section we will show that the C*-algebras $\Red(A,S,\alpha)$ and $\Ess(A,S,\alpha)$ from  \Cref{defn our crossed products} are isomorphic to, respectively,   the  reduced  crossed product $A\rtimes_\red S$ from \cite[Definition~4.1]{BussExelMeyer2017} and, when $S$ is quasi-countable, the essential crossed product $A\rtimes_\ess S$ from \cite[Definition~4.4]{KwasniewskiMeyer2021}.
Since \cite{BussExelMeyer2017} and \cite{KwasniewskiMeyer2021} assume that $S$ is unital, we do so as well; since we are only interested in isomorphism classes, by \Cref{lem A embeds} this assumption results in no loss of generality. We denote the identity in $S$ by $1$ and note that $A$ embeds in $A\rtimes_\full  S$ via $a\mapsto a\delta_1$. We also use \cref{lem: equivalence in omega} to rewrite $[e,\pi]=[1,\pi]$ for all $[e,\pi]\in G(\hat{A},S,\beta)^{(0)}$.
We start with the background needed to understand the definitions of  $A\rtimes_\red S$ and $A\rtimes_\ess S$.

Let $H$ be a Hilbert space and $M\subseteq B(H)$. We denote by $M'$ the commutant \[M'=\{T\in B(H): TS=ST\text{\ for all\ }S\in M\}.\] Let $A$ be a C*-algebra. We write $A''$ for the enveloping von Neumann algebra of $A$. Thus $A''$ is the bicommutant of the image $\pi_u(A)$ of $A$ under the universal representation $\pi_u$ of $A$. 
The \emph{local multiplier algebra} $\Mloc(A)$ of $A$ is
the direct limit 
$
\varinjlim \MM(J)
$
of the multiplier algebras $\MM(J)$ over all essential ideals $J \subseteq A$; see \cite{Ara-Mathieu}.

Let $I$ be an ideal of $A$. Then the \emph{complement of $I$ in $A$} is the ideal
$
I^\perp \coloneqq \{ x \in A : xI = Ix = \{0\}\}
$. The   ideal $I \oplus I^\perp$ of $A$ is always essential.

\begin{thm}\label{thm ER and EL} Let $\alpha\colon S \curvearrowright A$ be an action of a unital inverse semigroup $S$ on a C*-algebra $A$. 
    \begin{enumerate}
        \item\label{item: ER existence and properties} \textup{(See \cite[Lemma~4.5]{BussExelMeyer2017} or
        \cite[Proposition~3.16]{KwasniewskiMeyer2021})} For $s\in S$, write $1_s$ for the identity of $I_{s,1}''$ where $I_{s,1}''$ is viewed as an ideal of $A''$.   There exists a completely positive, linear and  contractive map \[\ER\colon A \rtimes_\full S \to A''\] such that $\ER(a_s\delta_s) = a_s 1_s$ for all $s \in S$ and $a_s \in A_{ss^*}$.
        \item\label{item: EL existence and properties} 
        \textup{(See \cite[Proposition~4.3]{KwasniewskiMeyer2021})}
        For $s\in S$, define $i_s \in \MM(I_{s,1}\oplus I_{s,1}^\perp)$ by $i_s(a\oplus b) \coloneqq a$ for $a \in I_{s,1}$ and $b \in I_{s,1}^\perp$.  There exists a completely positive, linear and contractive map \[\EL\colon A \rtimes_\full S \to \Mloc(A)\] such that $\EL(a_s \delta_s) = a_s i_s$ for all $s \in S$ and $a_s \in A_{ss^*}$.
    \end{enumerate}
\end{thm}

With $\ER$ and $\EL$ as in \Cref{thm ER and EL}, set
\begin{equation}\label{def N_E and N_EL}
\NN_\ER \coloneqq \{a \in A \rtimes_\full S : \ER(a^*a) = 0\} \quad \text{ and } \quad \NN_\EL \coloneqq \{a \in A \rtimes_\full S : \EL(a^*a) = 0\}.
\end{equation}
By \cite[Remark~4.5]{KwasniewskiMeyer2021}, $\NN_\ER$ and $\NN_\EL$ are ideals of $A \rtimes_\full S$ such that $\NN_\ER \subseteq \NN_\EL$.  The reduced crossed product $A\rtimes_\red S$ is defined in \cite[Definition~4.1]{BussExelMeyer2017}; here we will give a definition that is equivalent by \cite[Proposition~3.16]{KwasniewskiMeyer2021}.

\begin{defn} \label{defn: reduced and essential crossed products}
\textup{(See \cite[Proposition~3.16]{KwasniewskiMeyer2021} and \cite[Definition~4.4]{KwasniewskiMeyer2021})}
Let $\alpha\colon S \curvearrowright A$ be an action of a unital inverse semigroup on a C*-algebra $A$. The \emph{reduced crossed product} 
and the \emph{essential crossed product of $(A, S, \alpha)$} are 
\[
A \rtimes_\red S \coloneqq (A \rtimes_\full S) / \NN_\ER \quad \text{ and } \quad A \rtimesess S \coloneqq (A \rtimes_\full S) / \NN_\EL,
\]
respectively. 
\end{defn}

Let $\pi\colon A\to B(H)$ be a non-degenerate representation of $A$. By \cite[Theorem~3.7.7]{Pedersen-book} there exists a unique normal weakly-continuous extension $\pi''\colon A''\to \pi(A)''\subseteq B(H)$ of $\pi$.  For the proof of \Cref{prop: covariant rep implements expectation} we need the following lemma. 

\begin{lemma}\label{lemma: writing pi with projections}
  Let $I$ be an ideal of a C*-algebra $A$ and let $\pi\in \hat A$
  such that $\pi|_I\neq 0$.
  Write $1_I$ for the identity of $I''$. Then $\pi''(a)=\pi''(1_Ia)$
  for  $a\in A''$ and $\pi(a)=\pi''(1_Ia)$ for $a\in A$. 
\end{lemma}

\begin{proof}
    We note that $1_I$ is central in $A''$: for if $a\in A''$, then $1_Ia\in I''$ implies $1_Ia=1_Ia1_I=a1_I$. Define  $\psi\colon A''\to B(H_\pi)$ by $\psi(a)=\pi''(1_Ia)$. Since $1_I$ is central in $A''$ it follows that $\psi$ is a *-homomorphism. Notice that $\psi|_I=\pi|_I$.
    Since $I$ is an ideal of $A$ and $\pi$ is irreducible, $\pi|_{I}$ is non-degenerate.  Now $\psi|_{A}=\overline{\psi|_{I}}=\pi$ by the uniqueness of the extension.
    Finally, since $\psi$ is a normal weakly-continuous representation of $A''$ that extends $\pi$  we have $\psi=\pi''$. Now $\pi''(a)=\pi''(1_Ia)$ for all $a\in A''$, and then the last assertion follows. 
\end{proof}

Let $\pi\in \hat{A}$. The \emph{orbit} of $\pi$, denoted by $S\cdot\pi$, is  
$\{\beta_s(\pi): s\in S\}$.

\begin{prop}\label{prop: covariant rep implements expectation}
    Let $\alpha\colon S \curvearrowright A$ be an action of a unital inverse semigroup $S$ on a C*-algebra $A$. Let $a\in A\rtimes_\full S$.
    Let $\pi \in \hat{A}$ and  $h\in H_\pi$, and let $\chi^h_{[1,\pi]}$ be as in \Cref{notation: point mass}. Then  
    \begin{equation}\label{eq pi'' of E}
    \pi''(\ER(a))h=\big(\rho_\pi\times\sigma_\pi(a)\chi^h_{[1,\pi]}\big)([1,\pi]).
    \end{equation}
    Furthermore, $\rho_\pi\times\sigma_\pi(a)=0$ if and only if $\eta''(\ER(a^*a))=0$ for all $\eta\in S\cdot\pi$.
\end{prop}

\begin{proof}[Proof of \Cref{prop: covariant rep implements expectation}]
By continuity it suffices to prove \Cref{eq pi'' of E} for $a\in A\rtimes_\alg S$, that is, for $a=\sum_{s\in F}a_s\delta_s$ for a
 finite subset $F$ of $S$, and by linearity it suffices to consider $a=a_s\delta_s$ for a single $s\in S$. 
Since $1$ is a projection,  $I_{1,s}=I_{1,s^*}\subseteq A_{ss^*}\cap A_{s^*s}$. Now by \Cref{lem: equivalence in omega} we have $[1,\pi]=[s^*,\pi]$ if and only if $\pi\in\hat I_{1,s^*}=\hat I_{1,s}$ if and only if $[1,\pi]=[s,\pi]$. 
Thus
    \begin{align*}
        \big(\rho_\pi\times\sigma_\pi(a_s\delta_s)\chi^h_{[1,\pi]}\big)([1,\pi])
        &=\pi(a_s)(\sigma_\pi(s)\chi^h_{[1,\pi]})([1,\pi])\\
        &=\begin{cases}
            \pi(a_s)\chi^h_{[1,\pi]}([s^*,\pi]) &\text{if $\pi\in \hat A_{ss^*}$}\\
            0&\text{else}
        \end{cases}\\
        &=\begin{cases}
            \pi(a_s)h &\text{if $\pi\in \hat I_{1,s}$}\\
            0&\text{else}
        \end{cases}\\
         &=\begin{cases}
            \pi''(1_{I_{1,s}}a_s)h &\text{if $\pi\in \hat I_{1,s}$}\\
            0&\text{else}
        \end{cases}
        \end{align*}
 by  \Cref{lemma: writing pi with projections}.
 Since $\pi\in \hat I_{1,s}$ if and only if $\pi|_{I_{1,s}}\neq 0$ if and only if $\pi''(1_{I_{1,s}})\neq 0$ we see that 
\[(\rho_\pi\times\sigma_\pi(a_s\delta_s)\chi^h_{[1,\pi]})([1,\pi])=\pi''(1_{I_{1,s}}a_s)h=\pi''(\ER(a_s\delta_s))h\] as needed. 
Thus \Cref{eq pi'' of E} holds for all $a\in A\rtimes_\full S$.

To prove the second statement we 
adapt the argument from
\cite[Proposition~4.2.6]{Sims-Szabo-Williams}. 
Fix $a\in A\rtimes_\full S$. 
Fix $\pi\in \hat{A}$ and suppose that $\eta''(\ER(a^*a))=0$ for all $\eta\in S\cdot\pi$. Suppose, towards a contradiction, that  $\rho_\pi\times \sigma_\pi(a)\ne 0$. Since a dense subspace of  $\ell^2(G_\pi,H_\pi)$ is spanned by point masses $\chi^h_{[s,\pi]}$ there exist $h\in H_\pi$ and $s\in S$  such that  $\rho_\pi\times \sigma_\pi(a)\chi^h_{[s,\pi]}\neq 0$. 
Let $U_{[s,\pi]}\colon \ell^2(G_{\pi},H_{\pi})\to \ell^2(G_{\beta_s(\pi)},H_\pi)$ be the unitary from \cref{prop: covariant reps are unitarily equivalent}. Then 
$U_{[s,\pi]}\chi^h_{[s,\pi]}=\chi^h_{[ss^*,\beta_s(\pi)]}=\chi^h_{[1,\beta_s(\pi)]}$ using \Cref{lem: equivalence in omega}, and $U_{[s,\pi]}\rho_\pi\times \sigma_\pi(a)U^*_{[s,\pi]}=\rho_{\beta_s(\pi)}\times \sigma_{\beta_s(\pi)}(a)$. 
Now 
\begin{align*}
0&\neq \big(\rho_\pi\times \sigma_\pi(a^*a)\chi^h_{[s,\pi]}\mid \chi^h_{[s,\pi]}\big)\\
&=\big(\rho_{\beta_s(\pi)}\times \sigma_{\beta_s(\pi)}(a^*a)\chi^h_{[1,\beta_s(\pi)]}\mid \chi^h_{[1,\beta_s(\pi)]}\big)\\
&=\big(\rho_{\beta_s(\pi)}\times \sigma_{\beta_s(\pi)}(a^*a)\chi^h_{[1,\beta_s(\pi)]}([1,\beta_s(\pi)])\mid h\big)\\
&=\big(\beta_s(\pi)''(\ER(a^*a))h\mid h\big)
\end{align*}
by the above.
Thus $\beta_s(\pi)''(\ER(a^*a))\neq 0$, which is a contradiction since $\beta_s(\pi)\in S\cdot\pi$.  Hence $\rho_\pi\times \sigma_\pi(a)= 0$, as needed. 

   Next, suppose that $\rho_\pi\times \sigma_\pi(a)=0$.  Then for all $s\in S$ and $\beta_s(\pi)\in S\cdot\pi$ and $h\in H_{\beta_s(\pi)}$ we have 
  \begin{align*}
      \beta_s(\pi)''(\ER(a^*a))h&=\rho_{\beta_s(\pi)}\times \sigma_{\beta_s(\pi)}(a^*a)\chi^h_{[1,\beta_s(\pi)]}([1,\beta_s(\pi)])\\
      &=U_{[s,\pi]}\rho_\pi\times \sigma_\pi(a^*a)U^*_{[s,\pi]}\chi^h_{[1,\beta_s(\pi)]}([1,\beta_s(\pi)])=0.
  \end{align*}
   Thus $\eta''(\ER(a^*a))=0$ for all  $\eta\in S\cdot\pi$.
\end{proof}

\subsection{Reconciling the reduced crossed products}

\begin{thm}\label{thm reconciliation for reduced}
    Let $\alpha\colon S \curvearrowright A$ be an action of a unital inverse semigroup $S$ on a C*-algebra $A$. Then $\NN_\ER=\{a\in A\rtimes_\full S: \rho_\pi\times \sigma_\pi(a)=0\text{\ for all }\pi\in\hat A\}$ and  $\Red(A,S,\alpha)= A\rtimes_\red S$. 
\end{thm}

\begin{proof} Temporarily set $\Ii_\red\coloneqq \{a\in A\rtimes_\full S: \rho_\pi\times \sigma_\pi(a)=0\text{\ for all }\pi\in\hat A\}$. 
Since $\Red(A,S,\alpha)$ is  the completion of $A\rtimes_\alg S$ in the norm $\|\cdot\|_\red$, it follows that $\Red(A,S,\alpha)$ is the quotient of $A\rtimes_\full S$ by the ideal $\Ii_\red$ of $A\rtimes_\full S$. Thus 
   to see that $\Red(A,S,\alpha)= A\rtimes_\red S$ it suffices to show that 
   $\NN_\ER=\Ii_\red$.

     Let $a\in \NN_\ER$. Then $\ER(a^*a)=0$ and so $\pi''(\ER(a^*a))=0$ for all $\pi\in\hat{A}$. Now $\rho_\pi\times\sigma_\pi(a)=0$ for all $\pi\in\hat{A}$  by  \cref{prop: covariant rep implements expectation}. Thus  $a\in \Ii_\red$. 

     For the other direction, let $a\in \Ii_\red$. By \cref{prop: covariant rep implements expectation} we have $\pi''(\ER(a^*a))=0$ for all $\pi\in \hat A$.
     Recall from \cite[Definition~4.9]{BussExelMeyer2017} that a family of representations $\{\pi_i\}$ of $A$ is ``$\ER$-faithful'' if the direct sum of the restrictions of $\pi_i''$ to the range $\ER(A\rtimes_\full S)\subseteq A''$ of $\ER$ is faithful. Since  $\hat A$ is $\ER$-faithful by \cite[Theorem~7.4]{BussExelMeyer2017} we have $\ER(a^*a)=0$, giving $a\in\NN_\ER$. Now $\Ii_\red=\NN_\ER$. We conclude that $\Red(A,S,\alpha)= A\rtimes_\red S$.   
\end{proof}

Let $\alpha\colon S \curvearrowright A$ be an action of a  unital inverse semigroup $S$ on a C*-algebra $A$. Then $\alpha$ extends to an action  $\alpha''\colon S \curvearrowright A''$.
On page~243 of  \cite{BussExelMeyer2017}, it is also observed that the norm of $A\rtimes_\red S$ is the supremum of  norms in $\Ind\pi''$, where $\pi''$ is the extension to $A''$ of a representation $\pi$ of $A$. The induction is via  a right-Hilbert $(A''\rtimes_\red S)$-$A''$ bimodule $\ell^2(S, A'')$.  This, and that it suffices to take irreducible representations of $A$, is proved over several sections, culminating in \cite[Proposition~4.10 and Theorem~7.2]{BussExelMeyer2017}. In \Cref{prop reconcile induced reps} below we show that if $\pi\in\hat A$, then $(\Ind\pi'')|_{A\rtimes_\red S}$ is unitarily equivalent to the representation $\rho_\pi\times\sigma_\pi$  defined in \Cref{thm: rep induced from hat A} which factors through the reduced crossed product by \Cref{prop: reduced norm is a norm}.

 Let $\ER\colon A''\rtimes_\alg S\to A''$ be  as in  \Cref{thm ER and EL}.  Then $A''$ acts on the right of $A''\rtimes_\alg S$ by multiplication (after identifying $A''$ with $A''\delta_1$), and $\langle x\,,\, y\rangle=\ER(x^*y)$ is an $A''$-valued inner product on $A''\rtimes_\alg S$ by \cite[Proposition~3.6]{BussExelMeyer2017}. Then $\ell^2(S, A'')$ is by definition the completion of $A''\rtimes_\alg S$  in the norm induced by the inner product.  Furthermore, the left action of $A\rtimes_\alg S$ on $A''\rtimes_\alg S$  by multiplication extends, by definition of $A\rtimes_\red S$ in \cite[Definition~4.1]{BussExelMeyer2017}, to an action of $A\rtimes_\red S$ by adjointable operators on $\ell^2(S, A'')$.

Let $\pi\in\hat A$, and let $(\rho_\pi,\sigma_\pi)$ be the covariant representation of $\alpha\colon S \curvearrowright A$ on $\ell^2(G_\pi,H_\pi)$ of \Cref{thm: rep induced from hat A}.
To simplify matters, we will now view $(\rho_\pi,\sigma_\pi)$ on $\ell^2(G_\pi)\otimes H_\pi$  so that with  $1_{[t,\pi]}$   the point mass at $[t,\pi]\in G_\pi$ we have
\begin{align*}
\rho_\pi(a)(1_{[t,\pi]}\otimes h)&=1_{[t,\pi]}\otimes \beta_t(\pi)(a)h;\\
\sigma_\pi(s)(1_{[t,\pi]}\otimes h)&=\begin{cases}
1_{[st,\pi]}\otimes h&\text{if $\pi \in \hat{A}_{(st)^*st}$}\\0&\text{else.}
\end{cases}
\end{align*}

\begin{prop}\label{prop reconcile induced reps}
Let $\pi \in \hat{A}$ and view $(\rho_\pi,\sigma_\pi)$ on $\ell^2(G_\pi)\otimes H_\pi$, and let $\ell^2(S,A'')\otimes_{\pi''}H_\pi$ be the Hilbert space of the induced representation $\Ind\pi''$.  There is a unitary \[U\colon \ell^2(S,A'')\otimes_{\pi''} H_\pi\to \ell^2(G_\pi)\otimes H_\pi\] such that
$U(a_s \delta_s\otimes h)=1_{[s,\pi]}\otimes \pi''(\alpha_s''(a_s))h$ for $s\in S$, $a_s\in A_{ss^*}''$ and $h\in H_\pi$. Furthermore,  $U$ implements a unitary equivalence between $(\Ind \pi'')|_{A\rtimes_\red S}$ and $\rho_\pi\times\sigma_\pi$.
\end{prop}

\begin{proof}
Let $\pi\in \hat A$, say $\pi\colon A\to B(H_\pi)$. 
Note that  $\rho_\pi\times\sigma_\pi$    factors through the reduced crossed product by \Cref{prop: reduced norm is a norm}.
For $s\in S$ we have $\pi\in \hat A_{s^*s}$ if and only if $\pi''\in (A_{s^*s}'')^\wedge$.  Thus we identify the discrete fibres $G_\pi=\{[t,\pi]:t\in S \text{ and }\pi\in \hat A_{t^*t}\}$ and $G_{\pi''}$.

For finite subsets $F\subseteq S$, the map $\big(\sum_{s\in F} a_s\delta_s, h\big)\mapsto \sum_{s\in F}1_{[s,\pi]}\otimes \pi''(\alpha_s''(a_s))h$ from $(A''\rtimes_\alg S)\times H_\pi\to \ell^2(G_\pi)\otimes H_\pi$ is bilinear, and hence there is a well-defined function $U$ from the algebraic tensor product  $A''\rtimes_\alg S\odot H_\pi$ to $\ell^2(G_\pi)\otimes H_\pi$ such that $U(a_s\delta_s\odot h)=1_{[s,\pi]}\otimes \pi''(\alpha_s''(a_s))h$.
To see that $U$ descends to a well-defined map on $A''\rtimes_\alg S\otimes_{\pi''} H_\pi$ it suffices to check that $U$ is isometric for the balanced inner product.  

Let $s\in S$ and write $u_s$ for the identity in $A_{ss^*}''$.   To make our calculations easier, we first observe that 
\begin{equation}\label{eq dense subspace of induced} \vecspan\{u_s\delta_s\otimes_{\pi''} h : h\in H_\pi, s\in S \text{ such that } \pi \in \hat{A}_{ss^*}\}\end{equation} 
is  dense in  $\ell^2(S,A'')\otimes_{\pi''} H_\pi$.
To see this, note that 
\begin{align*}
u_s\delta_s\cdot \alpha''_{s^*}(a_s)
&=1_s\delta_s\alpha''_{s^*}(a_s)\delta_1
= \alpha''_s(\alpha''_{s^*}(u_s)\alpha''_{s^*}(a_s))\delta_{1s}
=\alpha''_s\circ\alpha''_{s^*}(u_sa_s)\delta_s=a_s\delta_s
\\\intertext{and that}
a_s\delta_s\otimes_{\pi''} h&= u_s\delta_s\cdot \alpha_{s^*}''(a_s)\otimes_{\pi''} h=u_s\delta_s \otimes_{\pi''} \pi''(\alpha''_{s^*}(a_s))h.
\end{align*}
Next, suppose that $s\in S$ such that  $\pi\not \in \hat{A}_{s^*s}$. Then $\pi|_{A_{s^*s}}=0$ and so $\pi''|_{A_{s^*s}''}=0$, giving $u_s\delta_s \otimes \pi''(\alpha_{s^*}(a_s))h
=0$. 
Since $\vecspan\{a_s\delta_s\otimes_{\pi''} h\colon h\in H_\pi, s \in S, a_s \in A_{ss^*}''\}$ is dense
in 
$\ell^2(S,A'')\otimes_{\pi''} H_\pi$, so is the subspace at \Cref{eq dense subspace of induced}.

Now let  $s,t \in S$ such that $\pi\in\hat A_{ss^*}\cap \hat A_{tt^*}$, and let $h,k \in H_\pi$. Then 
\begin{align*}
\big(U(u_s\delta_s\odot
h)\mid U(u_t\delta_t\odot
k)\big)
&=(1_{[s,\pi]}\otimes h\mid1_{[t,\pi]}\otimes k)\\
&=(1_{[s,\pi]}\mid 1_{[t,\pi]})(h\mid k)\\
&=\begin{cases}
(h\mid k)&\text{if $[s,\pi]=[t,\pi]$}\\
0&\text{otherwise.}
\end{cases}
\end{align*}
We have $(u_t\delta_t)^*u_s\delta_s=\alpha_{t^*}(u_tu_s)\delta_{t^*s}$ and $\ER(\alpha_{t^*}(u_tu_s)\delta_{t^*s})=\alpha_{xxt^*}(u_tu_s)1_{t^*s}$, where $1_{t^*s}$ is the identity in $(I_{t^*s, 1})''$.
Then $\alpha_{t^*}''(u_tu_s)$ is the identity in $A_{t^*ss^*t}$,  and if $\pi \in (I_{t^*s, 1})^\wedge$ then $\pi''(1_{t^*s})$ is the identity operator on $H_\pi$. Thus
\begin{align*}
(u_s\delta_s\odot
h\mid u_t\delta_t\odot  
k)
&=\big(\pi''(\langle u_t\delta_t\mid u_s\delta_s\rangle_{A''})h\mid k\big)\\
&=\big(\pi''(\ER((u_t\delta_t)^*u_s\delta_s))h\mid k\big)\\
&=\big(\pi''(\ER(\alpha_{t^*}''(u_tu_s)\delta_{t^*s}))h\mid k\big)\\
&=\big(\pi''(\alpha_{t^*}''(u_tu_s)1_{t^*s})h\mid k\big)\\
&=\big(\pi''(1_{t^*s})h\mid k\big)\\
&=\begin{cases}
(h\mid k) &\text{if $\pi \in  (I_{t^*s, 1})^\wedge$}\\
0 &\text{otherwise.}
\end{cases}
\end{align*}
By  \Cref{lem: equivalence in omega} we have 
\[
[s,\pi]=[t,\pi]\Longleftrightarrow [t^*s,\pi]=[1,\pi]\Longleftrightarrow\pi\in (I_{t^*s,1})^\wedge, 
\]
and we conclude that $\big(U(u_s\delta_s\odot h)\mid U(u_t\delta_t\odot k)\big)=(u_s\delta_s\odot
h\mid u_t\delta_t\odot k)$.
Thus $U$ descends to an isometry of $A''\rtimes_\alg S\otimes_{\pi''} H_\pi$ to $\ell^2(G_\pi)\otimes H_\pi$, and then extends to an isometry of $\ell^2(S,A'')\otimes_{\pi''} H_\pi$ to $\ell^2(G_\pi)\otimes H_\pi$. Finally notice that  $1_{[s,\pi]}\otimes h=U(u_s\delta_s\otimes_{\pi''}h)$ so that $U$ is surjective, hence is a unitary.

To show that $U$ intertwines $(\Ind \pi)''|_{A\rtimes_\red S}$ and $\rho_\pi\times\sigma_\pi$ we start with a side calculation involving the action of  $A\rtimes_\red S$ on $\ell^2(S,A'')\otimes_{\pi''} H_\pi$.  Let $s\in S$ and $a_s\in A_{ss^*}$. We have
\begin{align}\label{eq side calculation for intertwine}
  a_s\delta_s u_t\delta_t
  &=
  \alpha_{s}''\big(\alpha_{s^*}(a_s)u_t\big)\delta_{st}\notag\\
  &=\alpha_{st}''\big( \alpha_{(st)^*}''(u_{st})\alpha_{(st)^*}''\big( \alpha_s''(\alpha_{s^*}(a_s)u_t) \big)\big)\delta_{st}\notag\\
 &= u_{st}\alpha_s''(\alpha_{s^*}(a_s)u_t)\delta_{st}
\notag\\
&=
u_{st}\delta_{st}\alpha_{(st)^*}''\big(\alpha_s''\big(\alpha_{s^*}(a_s)u_t\big)\big)\delta_1\notag\\
&=u_{st}\delta_{st}\alpha_{t^*}''\big(\alpha_{s^*}(a_s)u_t\big)\delta_1
\end{align}
Then using \eqref{eq side calculation for intertwine} at the second step we have
\begin{align*}
\big(U\circ \Ind \pi''(a_s\delta_s)\big)(u_t\delta_t\otimes_{\pi''} h) &=U((a_s\delta_su_t\delta_t)\otimes_{\pi''} h)\\
&=U(u_{st}\delta_{st}\otimes_{\pi''} \pi''(\alpha_{t^*}''(\alpha_{s^*}(a_s)u_t))h)\\
&=\begin{cases}
1_{[st,\pi]}\otimes\pi''(\alpha_{t^*}''(\alpha_{s^*}(a_s)u_t))h &\text{if $\pi \in \hat A_{(st)^*st}$ }\\
0&\text{otherwise.}
\end{cases}
\end{align*}
On the other hand,  
\begin{align*}
\big(\rho_\pi\times\sigma_\pi(a_s\delta_s)\circ U \big)(u_t\delta_t\otimes_{\pi''} h)
&=\rho_\pi\times\sigma_\pi(a_s\delta_s)(1_{[t,\pi]}\otimes h)\\
&=\rho_\pi(a_s)\sigma_\pi(s)(1_{[t,\pi]}\otimes h)\\
&=\begin{cases}
\rho_\pi(a_s)(1_{[st,\pi]}\otimes h)&\text{if $\pi \in \hat{A}_{(st)^*st}$}\\
0 &\text{otherwise.}
\end{cases}
\end{align*}
Let $\pi \in \hat{A}_{(st)^*st}$. Then $\rho_\pi(a_s)(1_{[st,\pi]}\otimes h)=1_{[st,\pi]}\otimes \beta_{st}(\pi)(a_s)h$ and 
\[
  \beta_{st}(\pi)(a_s)=\overline{\pi\circ\alpha_{t^*}}(\alpha_{s^*}(a_s))=\overline{\pi\circ\alpha_{t^*}}''(\alpha_{s^*}(a_s)u_t)=\pi''\big(\alpha_{t^*}''(\alpha_{s^*}(a_s)u_t)).
\]
Thus 
\begin{align*}
\big(\rho_\pi\times\sigma_\pi(a_s\delta_s)\circ U \big)(u_t\delta_t\otimes_{\pi''} h)
&=\begin{cases}
(1_{[st,\pi]}\otimes \pi''(\alpha_{t^*}''(\alpha_{s^*}(a_s)u_t))h)&\text{if $\pi \in \hat{A}_{(st)^*st}$}\\
0 &\text{otherwise}
 \end{cases}\\
 &=\big(U\circ \Ind \pi''(a_s\delta_s)\big)(u_t\delta_t\otimes_{\pi''} h) 
\end{align*}
It follows that $(\Ind \pi'')|_{A\rtimes_\red S}$ and $\rho_\pi\times\sigma_\pi$ are unitarily equivalent representations of $A\rtimes_\red S$.
\end{proof}

\subsection{Reconciling the essential crossed products}
An inverse semigroup  $S$ is called \emph{quasi-countable}, as in \cite{CMS},    if there exists a countable set $K\subseteq S$ such that $S=KE(S)$. 
We now show that our $\Ess(A,S,\alpha)$ coincides with $A\rtimes_\ess S$ if $S$ is \emph{quasi-countable}. 

For a subset $U\subseteq \hat{A}$ we write $ S\cdot U\coloneq\cup_{\pi\in U} S\cdot\pi$ for the \emph{saturation} of $U$ under the action of $S$.

\begin{lemma}\label{lem: orbits are meagre}
     Let $\alpha\colon S \curvearrowright A$ be an action of a quasi-countable and unital inverse semigroup $S$ on a C*-algebra $A$. 
    \begin{enumerate}
        \item\label{item1: orbits are meagre} If $R\subseteq \hat{A}$ is  meagre, then its saturation $ S\cdot R$ is meagre.
        \item\label{item2: orbits are meagre}  If $R\subseteq \hat{A}$ is  comeagre, then there exists a comeagre  $V\subseteq R$ such that $V$ is invariant under the action of $S$.
    \end{enumerate}
\end{lemma}
\begin{proof}
    (\Cref{item1: orbits are meagre}) Fix a meagre $R\subseteq \hat{A}$. Write $R\subseteq \cup_{n\in \N} W_n$ where each $W_n$ is nowhere dense in $\hat{A}$.
     Since $S$ is quasi-countable, there is a countable subset $K$ such that $S=KE(S)$. Now  \[ S\cdot R=\bigcup_{t\in K}\beta_t(R\cap \hat{A}_{t^*t})\subseteq \bigcup_{t\in K}\bigcup_{n\in \N}\beta_t(W_n\cap \hat{A}_{t^*t})\]
     is a countable union. 
     Since each $\beta_t$ is a homeomorphism on $\hat{A}_{t^*t}$, each $\beta_t(W_n\cap \hat{A}_{t^*t})$ is nowhere dense in the relative topology in $\hat{A}_{tt^*}$. Then  $\beta_t(W_n\cap \hat{A}_{t^*t})$ is also nowhere dense in $\hat{A}$ by \cite[Theorem 11.5.4]{NariciBeckenstein2011}. Thus $ S\cdot R$  is meagre.

  (\Cref{item2: orbits are meagre}) Suppose that $R\subseteq\hat A$ is comeagre. Then $\hat{A}\setminus R$ is meagre and by (\Cref{item1: orbits are meagre}) its saturation is also meagre. Thus $V\coloneq\hat{A}\setminus S\cdot(\hat{A}\setminus R)$ is comeagre and contained in $R$; as the complement of a saturation  it is invariant.
\end{proof}

\begin{thm}\label{thm reconcile ess}
    Let $\alpha\colon S \curvearrowright A$ be an action of a 
    unital inverse semigroup $S$ on a C*-algebra $A$. 
    Let $\|\cdot\|_\ess$ be the seminorm of \Cref{eq ess norm} extended to $A\rtimes_\full S$. Then for $a\in A\rtimes_\full S$ we have 
\begin{equation}\label{eqn: ess min formula}
        \|a\|_\ess=\min\Big\{\sup_{\pi\in R}\|\rho_\pi\times \sigma_\pi(a)\|\colon R\subseteq \hat{A} \text{ is comeagre}\Big\}.
    \end{equation}
Furthermore, 
    $\{a\in A\rtimes_\full S: \|a\|_\ess=0\}\subseteq \NN_\EL$, with equality if $S$ is quasi-countable. If $S$ is quasi-countable, then    $\Ess(A,S,\alpha)=A\rtimes_\ess S$.
\end{thm}

\begin{proof}
We begin by showing that the infimum in the definition of $\|a\|_\ess$ is  a minimum, giving \Cref{eqn: ess min formula}.
For every $n\in \N$ there exists a comeagre subset $R_n\subseteq \hat{A}$ such that \[\|a\|_\ess\le \sup_{\pi\in R_n}\|\rho_\pi\times \sigma_\pi(a)\|<\|a\|_\ess +1/n.\]
   A countable intersection of comeagre sets is comeagre, and so $\cap_{n\in \N} R_n$ is comeagre. For each $m\in \N$ we have \[\|a\|_\ess\le \sup_{\pi\in \bigcap_{n\in \N}R_n}\|\rho_\pi\times \sigma_\pi(a)\|<\|a\|_\ess +1/m\]
and hence  $\cap_{n\in \N} R_n$ attains the infimum. This gives \Cref{eqn: ess min formula}.

Now the key observation is \cite[Corollary 4.12]{KwasniewskiMeyer2021}:\begin{equation}\label{eq key obs KM}\NN_\EL=\{a\in A\rtimes_\full S:\{\pi\in \hat{A}:\pi''(\ER(a^*a))=0\}\text{ is comeagre}\}.\end{equation}
    Fix $a\in A\rtimes_\full S$ and  suppose that $\|a\|_\ess=0$. Then by \Cref{eqn: ess min formula} there exists a comeagre  $R\subseteq \hat{A}$ such that $\|\rho_\pi\times \sigma_\pi(a)\|=0$ for all $\pi \in R$. By 
    \Cref{prop: covariant rep implements expectation}, \[R\subseteq \{\pi\in\hat A:\pi''(\ER(a^*a))=0\}\], and since $R$ is comeagre so is $\{\pi\in\hat A:\pi''(\ER(a^*a))=0\}$. Thus $a\in \NN_\EL$.

Now assume that $S$ is quasi-countable and let  $a\in \NN_\EL$. Then  $\{\pi\in \hat{A}:\pi''(\ER(a^*a))=0\}$ is comeagre. Since $S$ is quasi-countable, by  \Cref{lem: orbits are meagre}(\Cref{item2: orbits are meagre}) there exists an invariant comeagre set $R\subseteq \{\pi\in \hat{A}:\pi''(\ER(a^*a))=0\}$.  Then for all $\pi\in R= S\cdot R$ we have $\|\rho_\pi\times\sigma_\pi(a)\|=0$ by \Cref{prop: covariant rep implements expectation}. Thus $\|a\|_\ess=0$.
    
The last statement now follows from the definitions of $\Ess(A,S,\alpha)$ and $A\rtimes_\ess S$.
\end{proof}

\begin{remark}
    Suppose that $\alpha\colon S \curvearrowright A$ is aperiodic. Then \cite[Proposition~6.8]{KwasniewskiMeyer2021} also gives a seminorm such that the C*-algebraic completion of $A\rtimes_\alg S$ is isomorphic to $A\rtimes_\ess S$. This seminorm is the infimum of all seminorms on $A\rtimes_\alg S$ whose restrictions to $A$ coincide with the norm on $A$. 
\end{remark}

When $S$ is not quasi-countable we have not been able to describe an explicit norm for $A\rtimes_\ess S$ in terms of the representations $\rho_\pi\times\sigma_\pi$, but we can describe $\NN_\EL$. 

\begin{prop}
    Let $\alpha\colon S \curvearrowright A$ be an action of a unital inverse semigroup $S$ on a C*-algebra $A$.  Let $\pi \in \hat{A}$, $h\in H_\pi$, and let $\chi^h_{[1,\pi]}$ be as in \Cref{notation: point mass} and let $P_{[1,\pi]}$ be the projection onto $\clsp\{\chi^h_{[1,\pi]}:h\in H_\pi\}\subseteq \ell^2(G_\pi, H_\pi)$.  Then 
    \begin{equation}\label{eq NEL description}\NN_\EL=\{a\in A\rtimes_\full S: \{\pi\in \hat{A}:P_{[1,\pi]}(\rho_\pi\times \sigma_\pi(a^*a))P_{[1,\pi]}=0\}\text{ is comeagre}\}.\end{equation}
\end{prop}

\begin{proof}
Fix $a\in A\rtimes_\full S$. For $\pi\in\hat{A}$ and  $h\in H_\pi$ we have \[\pi''(\ER(a^*a))h=\big(\rho_\pi\times\sigma_\pi(a^*a)\chi^h_{[1,\pi]}\big)([1,\pi])\] by   \cref{prop: covariant rep implements expectation}.
For $\xi\in\ell^2(G_\pi,H_\pi)$ we have  $P_{[1,\pi]}\xi=\chi_{[1,\pi]}^{\xi([1,\pi])}$, and hence $P_{[1,\pi]}\xi=0$ if and only if $\xi([1,\pi])=0$. Thus
\begin{align*}
  \pi''(\ER(a^*a))=0&\Longleftrightarrow 
   \big(\rho_\pi\times\sigma_\pi(a^*a)\chi^h_{[1,\pi]}\big)([1,\pi])=0\text{ for all $h\in H_\pi$}\\
     &\Longleftrightarrow P_{[1,\pi]}(\rho_\pi\times\sigma_\pi(a^*a))\chi^h_{[1,\pi]} =0 \text{ for all $h\in H_\pi$}\\
     &\Longleftrightarrow P_{[1,\pi]}(\rho_\pi\times\sigma_\pi(a^*a))P_{[1,\pi]}=0.
\end{align*}
The result now follows from  the description of $\NN_\EL$ at \Cref{eq key obs KM}.
\end{proof}

\begin{remark}
If $S$ is not quasi-countable, \Cref{thm reconcile ess} only gives a sequence of quotient maps 
        \[
        A\rtimes_\red S\to \Ess(A,S,\alpha)\to A\rtimes_\ess S.
        \]
In some cases $A\rtimes_\red S\cong A\rtimes_\ess S$. This can occur, for example, when the unit space of $G(\hat{A},S,\beta)$ is closed, and the maps $\ER$ and $\EL$ become genuine conditional expectations into $A$ (see \cite[Proposition~3.20, Remark~4.6]{KwasniewskiMeyer2021}); When the unit space is not closed the reduced and essential crossed products can coincide as characterised in \cite[Corollary~4.17]{KwasniewskiMeyer2021}. See \cite[Theorem~5.22]{Clark-Exel-Pardo-Sims-Starling} for an \'etale groupoid where the essential and reduced groupoid C*-algebras coincide.
\end{remark}

\bibliographystyle{acm}
\makeatletter\renewcommand\@biblabel[1]{[#1]}\makeatother
\bibliography{references}

\end{document}

%% file: preamble.tex
\pdfoutput=1

\usepackage[utf8]{inputenc}
\usepackage[margin=2cm, headsep=0.75cm, footskip=0.75cm]{geometry}
\usepackage{amsmath, amssymb, mathtools, 
}
\usepackage[hang, flushmargin]{footmisc}
\usepackage[english]{isodate}
\usepackage{enumitem}
\usepackage[dvipsnames]{xcolor}
\usepackage{tikz}
\usepackage{tikz-cd}
\usepackage[noadjust]{cite}
\usepackage[colorlinks=true, linkcolor=blue, linkbordercolor=blue, citecolor=red, citebordercolor=red, urlcolor=indigo, linktocpage=true]{hyperref}
\usepackage[capitalise, nameinlink, noabbrev, nosort]{cleveref}
\usepackage{doi}
\usepackage[normalem]{ulem}

\usetikzlibrary{decorations.markings}

\definecolor{indigo}{HTML}{492DA5}

\allowdisplaybreaks

\setenumerate{listparindent=\parindent}

\makeatletter
\let\origsection\section
\renewcommand\section{\@ifstar{\starsection}{\nostarsection}}
\newcommand\sectionspace{\vspace{0.5ex}}
\newcommand\nostarsection[1]{\sectionspace\origsection{#1}\sectionspace}
\newcommand\starsection[1]{\sectionspace\origsection*{#1}\sectionspace}
\makeatother

\setlist[enumerate]{font=\normalfont}
\crefname{enumi}{}{}
\crefname{enumii}{}{}

\numberwithin{equation}{section}
\crefname{equation}{Equation}{Equations}

\crefname{inequality}{Inequality}{Inequalities}
\creflabelformat{inequality}{#2(#1)#3}

\newtheorem{theorem}{Theorem}[section]
\newtheorem{thm}[theorem]{Theorem}
\crefname{thm}{Theorem}{Theorems}

\newtheorem{lemma}[theorem]{Lemma}
\crefname{lemma}{Lemma}{Lemmas}

\newtheorem{prop}[theorem]{Proposition}
\crefname{prop}{Proposition}{Propositions}

\newtheorem{cor}[theorem]{Corollary}
\crefname{cor}{Corollary}{Corollaries}

\theoremstyle{definition}

\newtheorem{defn}[theorem]{Definition}
\crefname{defn}{Definition}{Definitions}

\newtheorem{remark}[theorem]{Remark}
\crefname{remark}{Remark}{Remarks}

\newtheorem{notation}[theorem]{Notation}
\crefname{notation}{Notation}{Notations}

\crefname{example}{Example}{Examples}

\newcommand{\N}{\mathbb{N}}

\newcommand{\II}{\mathcal{I}}
\newcommand{\Ii}{\mathcal{I}}
\newcommand{\JJ}{\mathcal{J}}

\newcommand{\MM}{\mathcal{M}}
\newcommand{\NN}{\mathcal{N}}

\newcommand{\BB}{\mathcal{B}}

\newcommand{\loc}{\mathrm{loc}}
\newcommand{\Mloc}{\MM_\loc}
\newcommand{\id}{\mathrm{id}}

\newcommand{\vecspan}{\operatorname{span}}

\newcommand{\clsp}{\overline{\vecspan}}

\newcommand{\dom}{\operatorname{dom}}
\newcommand{\ran}{\operatorname{ran}}

\newcommand{\alg}{\mathrm{alg}}
\newcommand{\ess}{\mathrm{ess}}
\newcommand{\Ess}{\mathrm{Ess}}
\newcommand{\full}{\mathrm{max}}
\newcommand{\red}{\mathrm{r}}
\newcommand{\Red}{\mathrm{Red}}

\newcommand{\rtimesess}{\rtimes_\ess}

\newcommand{\ER}{\mathrm{E}}
\newcommand{\EL}{\mathrm{EL}}
\newcommand{\Ind}{\operatorname{Ind}}
\newcommand{\Prim}{\operatorname{Prim}}
\newcommand{\Bis}{\operatorname{Bis}}

\newcommand{\inv}{^{-1}}